\documentclass[11pt,amsfonts]{amsart}
\usepackage{amssymb}
\usepackage{mathrsfs}
\usepackage[all,cmtip]{xy}
\usepackage{pstricks}

\numberwithin{equation}{section}

\newtheorem{thm}{Theorem}[section]
\newtheorem{defn}[thm]{Definition}
\newtheorem{prop}[thm]{Proposition}
\newtheorem{lemma}[thm]{Lemma}
\newtheorem{cor}[thm]{Corollary}

\newtheorem*{notation}{Notation}

\newcommand{\ds}{\displaystyle}
\newcommand{\inj}{\rightarrowtail}
\newcommand{\surj}{\twoheadrightarrow}
\newcommand{\embed}{\hookrightarrow}
\newcommand{\Bl}{{\rm Bl}}
\newcommand{\Sym}{{\rm Sym}}
\newcommand{\tor}{{\mathbf{t}}}
\newcommand{\fre}{{\mathbf{f}}}
\newcommand{\e}{\varepsilon}

\newcommand{\Gr}{{\rm Gr}}

\newcommand{\ul}{\underline}
\newcommand{\lt}{{\rm lt}}

\newcommand{\CD}{\xymatrix@R=1pc@C=1pc}
\newcommand{\CDR}{\xymatrix@R=1pc}
\newcommand{\CDC}{\xymatrix@C=1pc}

 \DeclareMathOperator{\Quot}{Quot}
 
\DeclareMathOperator{\EExt}{{\it\mathcal{E}\!xt}} \DeclareMathOperator{\Ext}{Ext}
 
\DeclareMathOperator{\HHom}{{\it\mathcal{H}\!om}} \DeclareMathOperator{\Hom}{Hom}
\DeclareMathOperator{\rank}{rank} \DeclareMathOperator{\Coker}{Coker}
 
\DeclareMathOperator{\PProj}{{\bf Proj}} \DeclareMathOperator{\Proj}{Proj}
\DeclareMathOperator{\Img}{Im} \DeclareMathOperator{\Mor}{Mor}

\def\2M{M}

\begin{document}

\title{The space of complete quotients}

\date{}
\author{Yi Hu}
\address{Department of Mathematics, University of Arizona, USA.}
\email{yhu@math.arizona.edu}
\author{Yijun Shao}
\email{yijun.shao@gmail.com}

\maketitle

\begin{abstract}
We introduce {\it complete quotients} over the projective line and
prove that they form  smooth projective varieties.    The resulting parameter spaces
coincide with the varieties constructed in \cite{HLS11} and \cite{Shao}.
Hence they provide modular smooth compactifications with normal crossing boundaries
of the spaces of algebraic maps from the projective line to Grassmannian varieties.
\end{abstract}

\section{Introduction}

We work throughout over an algebraically closed field $\Bbbk$ of characteristic 0.

Fix the Grassmannian $\Gr(k,V)$ of $k$-dimensional subspace in a vector space $V \cong \Bbbk^n$.
The set $\Mor_d(\mathbb{P}^1,\Gr(k, V))$ of degree $d$ algebraic maps from $\mathbb{P}^1$ to $\Gr(k,V)$
has a natural structure of a smooth  quasi-projective variety.
It comes with a natural compactification, the Grothendieck Quot scheme $Q_d:=\Quot^{d,n-k}_{V_{\mathbb{P}^1}/ \mathbb{P}^1/\Bbbk}$,
 parametrizing all equivalence classes  $[V_{\mathbb{P}^1}\surj F]$ of quotients of degree $d$ and rank $n-k$,
 where $V_{\mathbb{P}^1}:=V\otimes_{\Bbbk} \mathcal{O}_{\mathbb{P}^1}$.
The subset
\[
\mathring{Q}_d:=\{[V_{\mathbb{P}^1}\surj F]\in Q_d\,|\,F\text{ is locally free}\}
\]
 is  open in $Q_d$ and can be identified with the variety $\Mor_d(\mathbb{P}^1,\Gr(k, V))$.

The compactification $Q_d$ is smooth but the boundary $Q_d\setminus \mathring{Q}_d$ has rather intricate singularities.
It comes equipped with a natural filtration by closed subsets
\[
Z_{d,0} \subset Z_{d,1} \subset \cdots \subset Z_{d,d-1}=Q_d\setminus \mathring{Q}_d
\]
where $Z_{d,r}=\{[V_{\mathbb{P}^1}\surj F]\in Q_d\,|\,\text{the torsion of } F \text{ has degree}\geq d-r\}$, $r=0,\dots,d-1$.
The subsets $Z_{d,r}$ admit natural subscheme structures (cf. Section 3, \cite{Shao}).

\begin{thm}[\cite{HLS11,Shao}]\label{hls}
The singularities of the subscheme $Z_{d,r}$ can be resolved by repeatedly blowing up
$Z_{d,0}$, $Z_{d,1}, \cdots, Z_{d,r-1}$. Consequently, by iteratively blowing-up the Quot scheme $Q_d$
along  $Z_{d,0}$, $Z_{d,1}, \cdots, Z_{d,d-1}$, we obtain a smooth compactification
$\widetilde{Q}_d$   such that the boundary $\widetilde{Q}_d\setminus \mathring{Q}_d$ is a  simple normal crossing divisor.
\end{thm}

The authors of  \cite{HLS11} proved the case when $k=1$. In  his Ph.D thesis \cite{Shao},
the second named author proved it for all Grassmannians $\Gr(k,V)$, $k\ge 1$.

The main purpose of this paper is to provide $\widetilde{Q}_d$ the following modular interpretation.

For any coherent sheaf $X$ over $\mathbb{P}^1$, we let $X^\tor$ denote the torsion subsheaf of $X$
and $X^\fre =X/X^\tor$ denote the free part of $X$. For any coherent sheaves $A$ and $B$ on $\mathbb{P}^1$, an extension of $A$ by $B$ is a short exact sequence $0\to B\to X\to A\to 0$, which will also be shorthanded by $B\inj X\surj A$. Two extensions $B\inj X_1\surj A$ and $B\inj X_2\surj A$ differ by a scalar multiple $\lambda $ if there is an isomorphism $X_1\simeq X_2$ such that the following diagram
\[
\CD{
0\ar[r] & B\ar[r]\ar[d]_{\lambda\cdot} & X_1\ar[r]\ar[d]^{\simeq} & A\ar[r]\ar@{=}[d] & 0 \\
0\ar[r] & B\ar[r] & X_2\ar[r] & A\ar[r] & 0
}
\]
commutes, where $\lambda\cdot$ stands for the multiplication by the scalar $\lambda$. We use $[B\inj X\surj A]$  to denote the equivalence class of
the extension $B\inj X\surj A$ modulo scalar multiplication.

\begin{defn}
\label{def:cq} A complete quotient of $V_{\mathbb{P}^1}$ of degree $d$ and rank $n-k$ on $\mathbb{P}^1$ is either
\begin{itemize}
\item  a quotient $[V_{\mathbb{P}^1}\surj X_1]  \in Q_d$ such that $X_1$ is locally free;  or,
\item  a sequence $([V_{\mathbb{P}^1}\surj X_1], [X_1^\fre\inj X_2\surj X_1^\tor], \dots, [X_{m}^\fre\inj X_{m+1}\surj X_{m}^\tor])$ with $m\geq 1$ such that $[V_{\mathbb{P}^1}\surj X_1] \in Q_d$,  for every $1 \leq i \leq m$,   $X_{i+1}$ is a non-split extension of $X_{i}^\tor$ by $X_{i}^\fre$,
and further, the last sheaf $X_{m+1}$ is the unique one that is locally free.
\end{itemize}
\end{defn}

It follows from the definition that all the sheaves $X_1,\dots,X_{m+1}$  are coherent and have degree $d$ and rank $n-k$.

The main theorem of this paper is
\begin{thm}\label{hs}
The projective variety $\widetilde{Q}_d$ parameterizes all complete quotients of $V_{\mathbb{P}^1}$ of degree $d$ and rank $n-k$.
\end{thm}

To  prove this theorem,  we calculate the normal bundle of the blowup center
in  each blowup of $\widetilde{\mathcal Q}_d \longrightarrow Q_d$  as presented
in  \cite{HLS11,Shao}, and  modularly interpret the points in the normal bundle as the desired extensions.

This work is the genus zero case of a larger project. The higher genus case will appear in forthcoming papers.

\medskip
\noindent
{\bf Acknowledgement.}
While this paper being prepared, the first named author was partially supported by  NSF DMS 0901136.  He would also like to dedicate this work
to the memory of Andrey Todorov who passed away in Jerusalem in the spring of 2012.

The notations in this paper closely follow those in \cite{Shao} with occasional modifications.




\section{The space of relative extensions}

\subsection{The space of non-split extensions}

 Let $F$ and $T$ be two coherent sheaves over $\mathbb{P}^1$ with $F$ being locally free and $T$ torsion. The vector space $E=\Ext^1_{\mathbb{P}^1}(T,F)$ can be identified with the set of all (isomorphism classes of) extensions $F\inj X\surj T$ of  $T$ by $F$.
 The zero element  corresponds to the split extension $F\inj F\oplus T\surj T$ while the remainders correspond to non-split ones. Therefore the projective space $\mathbb{P}(E):=\Proj(\Sym^*(E^\vee))$,
 where $E^\vee$ is the linear dual of the vector space $E$,
 parametrizes the set of all non-split extensions of $T$ by $F$ up to scalar multiplication.
 Here again, as in the introduction,  two extensions $F\inj X_1 \surj T$ and $F\inj X_2\surj T$ differ by a nonzero scalar multiple $\lambda$ if and only if there is an isomorphism $X_1\simeq X_2$ that makes the following diagram commute:
\[
\CD{
0\ar[r] & F\ar[r]\ar[d]_{\lambda\cdot} & X_1\ar[r]\ar[d]^{\simeq} & T\ar[r]\ar@{=}[d] & 0 \\
0\ar[r] & F\ar[r] & X_2\ar[r] & T\ar[r] & 0
}
\]
where $\lambda\cdot$ is the scalar multiplication by $\lambda$. We denote by $[F\inj X\surj T]$ the point of $\mathbb{P}(E)$ corresponding to a non-split extension $F\inj X\surj T$. One checks that an extension $F\inj X\surj T$  splits if and only if $\deg X^\tor =\deg T$. For a discussion on universal  extensions, see Example 2.1.12,  \cite{Huy} .

We need to introduce a relative version of the space of non-split extensions. For this, we begin with a lemma.

\begin{lemma}\label{lemma-Ext}
Let $S$ be a noetherian scheme, $0\to \mathcal{E}_1\to \mathcal{E}_0\to \mathcal{T}\to 0$ a short exact sequence of coherent sheaves on $S$ with $\mathcal{T}$ torsion and $\mathcal{E}_0$ and $\mathcal{E}_1$ locally free, and $\mathcal{F}$ a coherent torsion-free sheaf on $S$. Then
\begin{enumerate}
\item $H^0(S,\EExt^1(\mathcal{T},\mathcal{F}))=\Ext^1(\mathcal{T},\mathcal{F})$;
\item for any morphism $f: R\to S$, $f^*\EExt^1_S(\mathcal{T},\mathcal{F})=\EExt^1_R(f^*\mathcal{T},f^*\mathcal{F})$.
\end{enumerate}
\end{lemma}
\begin{proof}
(1) Since $\mathcal{T}$ is torsion and $\mathcal{F}$ is torsion-free, we have $\Hom(\mathcal{T},\mathcal{F})=0$ and $\HHom(\mathcal{T},\mathcal{F})=0$.
Applying $\Hom(-,\mathcal{F})$ and $\HHom(-,\mathcal{F})$ to the locally free presentation of $\mathcal{T}$, we obtain a long exact sequence
\[
0\to \Hom(\mathcal{E}_0,\mathcal{F})\to \Hom(\mathcal{E}_1,\mathcal{F})\to \Ext^1(\mathcal{T},\mathcal{F})\to \Ext^1(\mathcal{E}_0,\mathcal{F})\to \Ext^1(\mathcal{E}_1,\mathcal{F})
\]
and a short exact sequence
\[
0\to \HHom(\mathcal{E}_0,\mathcal{F})\to \HHom(\mathcal{E}_1,\mathcal{F})\to \EExt^1(\mathcal{T},\mathcal{F})\to 0
\]
Taking global sections of the above short exact sequence, we obtain a long exact sequence:
\[
0\to H^0(\HHom(\mathcal{E}_0,\mathcal{F}))\to H^0(\HHom(\mathcal{E}_1,\mathcal{F}))\to H^0(\EExt^1(\mathcal{T},\mathcal{F}))\to H^1(\HHom(\mathcal{E}_0,\mathcal{F}))\to H^1(\HHom(\mathcal{E}_1,\mathcal{F}))
\]
We have identifications
\[
\Hom(\mathcal{E}_i,\mathcal{F})=H^0(\HHom(\mathcal{E}_i,\mathcal{F})),\quad i=0,1
\]
and a natural homomorphism
\[
\Ext^1(\mathcal{T},\mathcal{F})\to H^0(\EExt^1(\mathcal{T},\mathcal{F}))
\]
Using the Grothendieck spectral sequence for the composition of the two functors $H^0$ and $\HHom^1(\mathcal{E}_i,-)$, we obtain exact sequences of low degrees:
\[
0\to H^1(\HHom(\mathcal{E}_i,\mathcal{F}))\to \Ext^1(\mathcal{E}_i,\mathcal{F})\to H^0(\EExt^1(\mathcal{E}_i,\mathcal{F}))\to\cdots,\quad i=0,1.
\]
Since $\mathcal{E}_i$ are locally free, we have $H^0(\EExt^1(\mathcal{E}_i,\mathcal{F}))=0$. Therefore we obtain identifications:
\[
H^1(\HHom(\mathcal{E}_i,\mathcal{F}))= \Ext^1(\mathcal{E}_i,\mathcal{F}),\quad i=0,1
\]
Thus we have a commutative diagram
\[
\CD{
\Hom(\mathcal{E}_0,\mathcal{F})\ar[r]\ar@{=}[d]& \Hom(\mathcal{E}_1,\mathcal{F})\ar@{=}[d]\ar[r]& \Ext^1(\mathcal{T},\mathcal{F})\ar[r]\ar[d]& \Ext^1(\mathcal{E}_0,\mathcal{F})\ar[r]\ar@{=}[d]& \Ext^1(\mathcal{E}_1,\mathcal{F})\ar@{=}[d] \\
H^0(\HHom(\mathcal{E}_0,\mathcal{F}))\ar[r]& H^0(\HHom(\mathcal{E}_1,\mathcal{F}))\ar[r]& H^0(\EExt^1(\mathcal{T},\mathcal{F}))\ar[r]& H^1(\HHom(\mathcal{E}_0,\mathcal{F}))\ar[r]& H^1(\HHom(\mathcal{E}_1,\mathcal{F}))
}
\]
By the Five Lemma, we obtain the identification
\[
H^0(\EExt^1(\mathcal{T},\mathcal{F}))=\Ext^1(\mathcal{T},\mathcal{F})
\]
The proof of  (2) is parallel to that of  \cite{Shao}, Proposition 2.3. So, we omit the details.
\end{proof}

\subsection{The space of relative non-split extensions}

To proceed, we fix a set of notations.
\begin{notation}  For  any $\Bbbk$-schemes $R$ and $S$,
\begin{enumerate}
\item we denote the projection $\mathbb{P}^1\times R\to R$ by $\pi_R$ or simply $\pi$;

\item for any coherent sheaf $\mathcal{H}$ over $R$,
we let $\mathcal{H}^\vee$ denote the  dual sheaf $\HHom_R(\mathcal{H},\mathcal{O}_R)$;

\item for any coherent sheaf $\mathcal{F}$ on $\mathbb{P}^1\times R$,
we set $\mathcal{F}_x:=\mathcal{F}|_{\mathbb{P}^1\times\{x\}}$ for any point $x\in R$;

\item  for  any morphism $f: R\to S$, we  set
$\bar f:=1\times f: \mathbb{P}^1\times R\to \mathbb{P}^1\times S$.
\end{enumerate}
\end{notation}

We now introduce a relative version of the space of non-split extensions. Let $S$ be a noetherian scheme over $\Bbbk$ and let $\mathcal{F}$ and $\mathcal{T}$ be two coherent sheaves on $\mathbb{P}^1\times S$, both flat over $S$, with $\mathcal{F}$ locally free and $\mathcal{T}$ torsion (i.e., having rank 0). Let $\pi_S: \mathbb{P}^1\times S\to S$ be the projection, and set
\[
\mathcal{E}:=\pi_{S*}\EExt^1_{\mathbb{P}^1\times S}(\mathcal{T},\mathcal{F}).
\]
As in \cite{Shao}, Proposition 2.4 (1), one checks that $\EExt^1(\mathcal{T},\mathcal{F})$ is a torsion sheaf and is flat over $S$. By Cohomology and Base Change, we can show $\mathcal{E}$ is locally free and for any point $s\in S$,
\[
\mathcal{E}|_s=H^0(\EExt^1(\mathcal{T},\mathcal{F})_s)=H^0(\EExt^1(\mathcal{T}_s,\mathcal{F}_s))=\Ext^1(\mathcal{T}_s,\mathcal{F}_s).
\]
where the second equality holds by Lemma \ref{lemma-Ext} (2) and the third holds by Lemma \ref{lemma-Ext} (1). So the projective bundle
\[
\mathbb{P}(\mathcal{E}):=\PProj(\Sym^*(\mathcal{E}^\vee))
\]
over $S$ is a family of spaces of non-split extensions: for each point $s\in S$, the fiber of $\mathbb{P}(\mathcal{E})$ over $s$ is $\mathbb{P}(\mathcal{E}|_s)=\mathbb{P}(\Ext^1(\mathcal{T}_s,\mathcal{F}_s))$.

\subsection{The universal extension}

\begin{notation}
For any coherent sheaf $H$ on $\mathbb{P}^1\times \mathbb{P}(\mathcal{E})$, we denote by $H(m,n)$ the sheaf
$$H\otimes p^*\mathcal{O}_{\mathbb{P}^1}(m)\otimes \pi_{\mathbb{P}(\mathcal{E})}^*\mathcal{O}_{\mathbb{P}(\mathcal{E})}(n)$$
where $p: \mathbb{P}^1\times \mathbb{P}(\mathcal{E})\to \mathbb{P}^1$ and $\pi_{\mathbb{P}(\mathcal{\mathcal{E}})}: \mathbb{P}^1\times \mathbb{P}(\mathcal{E})\to  \mathbb{P}(\mathcal{E})$ are the two projections.
\end{notation}

Let $a: \mathbb{P}(\mathcal{E})\to S$ be the structure morphism. The space $\mathbb{P}(\mathcal{E})$ comes equipped with a universal quotient $a^*\mathcal{E}^\vee\surj \mathcal{O}_{\mathbb{P}(\mathcal{E})}(1)$. Dualizing the universal quotient and tensoring the result with $\mathcal{O}_{\mathbb{P}(\mathcal{E})}(1)$, we obtain a line subbundle
\[
\begin{split}
\mathcal{O}_{\mathbb{P}(\mathcal{E})} &\embed a^*\mathcal{E}\otimes\mathcal{O}_{\mathbb{P}(\mathcal{E})}(1)=\pi_{\mathbb{P}(\mathcal{E})*}\bar a^*\EExt^1(\mathcal{T},\mathcal{F}) \otimes\mathcal{O}_{\mathbb{P}(\mathcal{E})}(1)=\\
&\pi_{\mathbb{P}(\mathcal{E})*}(\EExt^1(\bar a^*\mathcal{T},\bar a^*\mathcal{F})(0,1))= \pi_{\mathbb{P}(\mathcal{E})*}\EExt^1(\bar a^*\mathcal{T},\bar a^*\mathcal{F}(0,1))
\end{split}
\]
Here the first equality holds because the morphism $a$ is flat, and the second holds by Lemma \ref{lemma-Ext} (2).
This line subbundle corresponds to a nonzero element of
\[
\begin{split}
\Gamma(\mathbb{P}(\mathcal{E}),\pi_{\mathbb{P}(\mathcal{E})*}\EExt^1(\bar a^*\mathcal{T},\bar a^*\mathcal{F}(0,1)))
&=\Gamma(\mathbb{P}^1\times \mathbb{P}(\mathcal{E}),\EExt^1(\bar a^*\mathcal{T},\bar a^*\mathcal{F}(0,1)))  \\
&=\Ext^1(\bar a^*\mathcal{T},\bar a^*\mathcal{F}(0,1))
\end{split}
\]
where the second equality holds by Lemma \ref{lemma-Ext} (1). We can write this element as an extension
\begin{equation}\label{universal-extension}
0\to \bar a^*\mathcal{F}(0,1)\to \mathcal{X}\to \bar a^*\mathcal{T}\to 0.
\end{equation}
which we call the \emph{universal extension}. Note that the universal extension is \emph{nowhere-split}, which means that, for each $\Bbbk$-point $s\in S$ and each $\Bbbk$-point $x\in a^{-1}(\{s\})=\mathbb{P}(\Ext^1(\mathcal{T}_s,\mathcal{F}_s))$, the extension
\begin{equation}\label{pullback-extension}
0\to \mathcal{F}_s\to \mathcal{X}_x\to \mathcal{T}_s\to 0
\end{equation}
obtained by pulling back the universal extension to $\mathbb{P}^1\times\{x\}\simeq \mathbb{P}^1$ is non-split.

Conversely, the universal quotient can be recovered from the universal extension by reversing the above process. 
In fact, the universal extension itself is a nonzero element of
\[
\begin{split}
&\Ext^1(\bar a^*\mathcal{T},\bar a^*\mathcal{F}(0,1)) = \Gamma(\mathbb{P}^1\times \mathbb{P}(\mathcal{E}),\EExt^1(\bar a^*\mathcal{T},\bar a^*\mathcal{F}(0,1)) \\
& = \Gamma(\mathbb{P}(\mathcal{E}),\pi_{\mathbb{P}(\mathcal{E})*}\EExt^1(\bar a^*\mathcal{T},\bar a^*\mathcal{F}(0,1))) =\Gamma(\mathbb{P}(\mathcal{E}),a^*\mathcal{E}(1))
\end{split}
\]
hence determines a line subbundle
$\mathcal{O}_{\mathbb{P}(\mathcal{E})}\embed 
a^*\mathcal{E}(1)$. Tensoring the line subbundle with $\mathcal{O}_{\mathbb{P}(\mathcal{E})}(-1)$ and taking the dual, we obtain the universal quotient.

Thus, based on the universal property of the universal quotient, we have
\begin{thm}\label{thm-universal}
Let $R$ be an $S$-scheme with structure morphism $\rho: R\to S$, $L$ a line bundle on $R$. If $0\to \bar \rho^*\mathcal{F}\otimes\pi_R^* L\to \mathcal{Y}\to \bar \rho^*\mathcal{T}\to 0$ is a nowhere-split extension on $\mathbb{P}^1\times R$, then there is a unique $S$-morphism $f: R\to \mathbb{P}(\mathcal{E})$ such that there are isomorphisms $L\stackrel{g}{\simeq} f^*\mathcal{O}_{\mathbb{P}(\mathcal{E})}(1)$, $\mathcal{Y}\stackrel{h}{\simeq} \bar f^*\mathcal{X}$ that make the following diagram commute:
\[
\CD{
0\ar[r] &\bar \rho^*\mathcal{F}\otimes \pi_R^*L\ar[r]\ar[d]^{\simeq}_{1\otimes \pi_R^*g} &\mathcal{Y}\ar[r]\ar[d]^{\simeq}_{h} &\bar r^*\mathcal{T}\ar[r]\ar@{=}[d] & 0 \\
0\ar[r] &\bar f^*\bar a^*\mathcal{F}\otimes \pi_R^*f^*\mathcal{O}_{\mathbb{P}(\mathcal{E})}(1)\ar[r] &\bar f^*\mathcal{X}\ar[r] &\bar f^*\bar a^*\mathcal{T}\ar[r] & 0 \\
}
\]
where the second row is the pullback of the universal extension via $\bar f$.
\end{thm}

\section{A flattening stratification}

In this section, we continue to follow the previous notations, and we impose an additional assumption on the sheaf $\mathcal{F}$ of  Theorem \ref{thm-universal} (also Lemma \ref{lemma-Ext}):
\begin{equation}\label{assumption}
R^1\pi_{S*}(\mathcal{F}\otimes p^*\mathcal{O}_{\mathbb{P}^1}(-1))=0,
\end{equation}
where $p: \mathbb{P}^1\times S\to \mathbb{P}^1$ is the projection.

We also need the following notation (cf. Introduction).

\begin{notation}
For any coherent sheaf $H$ on $\mathbb{P}^1$, we denote by $H^\tor$ the torsion subsheaf of $H$. The quotient $H/H^\tor$ is locally free and we will denote it by $H^\fre$. Since $H\simeq H^\tor\oplus H^\fre$ on $\mathbb{P}^1$, we call $H^\tor$ and $H^\fre$ the torsion part and the locally free part of $H$, respectively.
\end{notation}

Let $d_1$ and $d_2$ be the relative degrees of $\mathcal{F}$ and $\mathcal{T}$ over $S$, respectively. Then
the universal extention $\mathcal{X}$ (see (\ref{universal-extension}))
 is of relative degree $d:=d_1+d_2$ over $S$. The set
 $$\mathring{\mathbb{P}}(\mathcal{E}):=\{x\in \mathbb{P}(\mathcal{E})\,|\,\mathcal{X}_x\;  \hbox{is locally free}\}$$
 is an open subset of $\mathbb{P}(\mathcal{E})$. Its complement $\mathbb{P}(\mathcal{E})\setminus \mathring{\mathbb{P}}(\mathcal{E})$ has a sequence of nested closed subsets:
\[
\varnothing=Y_{d_1}\subset Y_{d_1+1}\subset \cdots \subset Y_{d-1}=\mathbb{P}(\mathcal{E})\setminus \mathring{\mathbb{P}}(\mathcal{E})
\]
where $d$ is the relative degree of $\mathcal{T}$ over $S$, and
\[
Y_r=\{x\in \mathbb{P}(\mathcal{E})\,|\,\deg((\mathcal{X}_x)^\tor) \geq d-r\},\quad r=d_1,\dots,d-1.
\]
Note that an extension $0\to \mathcal{F}_s \to X\to \mathcal{T}_s\to 0$ is split if and only if the torsion part of $X$ is of degree $d_2$. Hence $Y_{d_1}=\varnothing$.

For $m\geq 0$, applying $\Hom(-,\mathcal{O}(m))$ to { the extension (\ref{pullback-extension})}, we obtain an exact sequence
\[
0\to \Hom(\mathcal{X}_x,\mathcal{O}(m))\to \Hom(\mathcal{F}_s,\mathcal{O}(m))\stackrel{\delta_{x,m}}{\to} \Ext^1(\mathcal{T}_s,\mathcal{O}(m))\to \Ext^1(\mathcal{X}_x,\mathcal{O}(m))\to 0
\]
where $\delta_{x,m}$ denotes the connecting homomorphism. Then we have
\begin{prop}
$\rank\delta_{x,m}=d_2-\deg (\mathcal{X}_x)^\tor$ for any $m\geq d$.
\end{prop}
\begin{proof}
Suppose $\mathcal{F}_s$ has rank $k$. Then $\mathcal{X}_x$ has rank $k$ as well. We have non-canonical isomorphisms
\[
\mathcal{F}_s\simeq \bigoplus_{i=1}^r\mathcal{O}(a_i), \quad \mathcal{X}_x\simeq (\mathcal{X}_x)^\tor\oplus \bigoplus_{i=1}^r\mathcal{O}(b_i)
\]
where $\sum_{i=1}^r{a_i}=d_1$ and $\sum_{i=1}^r b_i=d-\deg (\mathcal{X}_x)^\tor$.
By the assumption (\ref{assumption}), we have $H^1(\mathbb{P}^1,\mathcal{F}_s(-1))=0$ and also $H^1(\mathbb{P}^1,\mathcal{X}_x(-1))=0$. It follows that $a_i\geq 0$ and $b_i\geq 0$ for all $i$. Hence we have $a_i\leq d_1< d$ and $b_i\leq d$ for all $i$. So
$$\Hom(\mathcal{F}_s,\mathcal{O}(m))\simeq \bigoplus_{i=1}^k\Hom(\mathcal{O}(a_i),\mathcal{O}(m))=\bigoplus_{i=1}^k H^0(\mathcal{O}(m-a_i)),\;\; \hbox{and}$$
$$\Hom(\mathcal{X}_x,\mathcal{O}(m))\simeq \bigoplus_{i=1}^k\Hom(\mathcal{O}(b_i),\mathcal{O}(m))=\bigoplus_{i=1}^k H^0(\mathcal{O}(m-b_i)).$$ For $m\geq d$, we have $a_i\leq m$ and $b_i\leq m$, therefore
\[
\begin{split}
\rank \delta_{x,m} &=\dim \Hom(\mathcal{F}_s,\mathcal{O}(m))-\dim\Hom(\mathcal{X}_x,\mathcal{O}(m))=\sum_{i=1}^k(m-a_i+1)-(m-b_i+1) \\
& = \sum_{i=1}^k b_i-\sum_{i=1}^k a_i = d-\deg(\mathcal{X}_x)^\tor-d_1 = d_2-\deg (\mathcal{X}_x)^\tor
\end{split}
\]
\end{proof}

\begin{cor}\label{cor_rank}
For any integers $m,r$ with $m\geq d$ and $d_1\leq r\leq d$, we have $\deg (\mathcal{X}_x)^\tor\geq d-r$ if and only if $\rank \delta_{x,m}\leq r-d_1$.
\end{cor}

\begin{notation}
For any scheme $R$ and any coherent sheaf $\mathcal{H}$ on $R$, we will use the following abbreviations:
\[
\mathcal{H}^\e :=\EExt^1_R(\mathcal{H},\mathcal{O}_R)
\]
\end{notation}

Corollary \ref{cor_rank} suggests us that we can define the scheme structure of $Y_r$ in the following way. Let $m\gg 0$. Applying $\HHom(-,\mathcal{O}_{\mathbb{P}^1\times \mathbb{P}(\mathcal{E})}(m,1))$ to the universal extension (\ref{universal-extension}) to obtain an exact sequence
\[
0\to \mathcal{X}^\vee(m,1)\to \bar a^*\mathcal{F}^\vee(m,0)\stackrel{\delta_m}{\to} \bar a^*\mathcal{T}^\e(m,1)\to \mathcal{X}^\e(m,1)\to 0
\]
where $\delta_m$ is the connecting homomorphism. Next, applying $\pi_{\mathbb{P}(\mathcal{E})*}$ to $\delta_m$ and using identifications
\[
\begin{split}
\pi_*(\bar a^*\mathcal{F}^\vee(m,0))&=\pi_*\bar a^*(\mathcal{F}^\vee(m))=a^*(\pi_*\mathcal{F}^\vee(m)),\\
\pi_*(\bar a^*\mathcal{T}^\e(m,1))&=\pi_*\bar a^*(\mathcal{T}^\e(m))(1)=a^*(\pi_*\mathcal{T}^\e(m))(1)
\end{split}
\]
we obtain a nowhere-vanishing homomorphism
\begin{equation}\label{pushed-hom}
\pi_*\delta_m: a^*(\pi_*\mathcal{F}^\vee(m))\to a^*(\pi_*\mathcal{T}^\e(m))(1)
\end{equation}
Note that both $\pi_*\mathcal{F}^\vee(m)$ and $\pi_*\mathcal{T}^\e(m)$ are locally free sheaves on $S$ for $m\gg 0$. Applying the exterior power $\bigwedge^{l+1}$ ($l\geq 0$) to $\pi_*\delta_m$, we obtain
\[
\bigwedge^{l+1}\pi_*\delta_m: \bigwedge^{l+1}a^*(\pi_*\mathcal{F}^\vee(m))\to \bigwedge^{l+1}a^*(\pi_*\mathcal{T}^\e(m))(l+1)
\]
By Corollary \ref{cor_rank}, $Y_{d_1+l}$ is exactly the (set-theoretic) zero locus of $\bigwedge^{l+1}\pi_*\delta_m$, for each $l$ with $0\leq l<d_2$. The section $\bigwedge^{l+1}\pi_*\delta_m$ induces a homomorphism
\[
\HHom\bigg(\bigwedge^{l+1}a^*(\pi_*\mathcal{F}^\vee(m)),\bigwedge^{l+1}a^*(\pi_*\mathcal{T}^\e(m))(l+1)\bigg)^\vee\to \mathcal{O}_{\mathbb{P}(\mathcal{E})}
\]
The image, which we denote by $I_{l,m}$, is an ideal sheaf. In a similar way as in \cite{Shao}, Proposition 3.4, we can prove that
\begin{prop}
There exists an integer $N>0$ such that $I_{l,m}=I_{l,N}$ as subsheaves of $\mathcal{O}_{\mathbb{P}(\mathcal{E})}$ for all $m>N$ and for all $l$ with $0\leq l<d_2$.
\end{prop}
Obviously, the subscheme defined by the ideal $I_{l,N}$ is supported on the closed subset $Y_{d_1+l}$. For simplicity, we denote this subscheme still by $Y_{d_1+l}$, that is,
\begin{equation}\label{subscheme}
Y_r = \text{ the closed subscheme defined by the ideal }I_{r-d_1+1,N}, \quad r=d_1+1,\dots,d-1
\end{equation}
In addition, we set $\mathring{Y}_r:= Y_r\setminus Y_{r-1}$ for $d_1<r<d$. Note that $\mathring{Y}_{d_1+1}=Y_{d_1+1}$ since $Y_{d_1}=\varnothing$.

Again, in a similar way as in \cite{Shao}, Theorem 3.5, we can prove that
\begin{prop}\label{FlatteningStratification}
The locally closed subschemes $Y_{d_1+1}$, $\mathring{Y}_{d_1+2}$, $\cdots$, $\mathring{Y}_{d-1}$ of $\mathbb{P}(\mathcal{E})$ form the flattening stratification of $\mathbb{P}(\mathcal{E})$ by the sheaf $\mathcal{X}^\e$, which means that, for any noetherian $\Bbbk$-scheme $R$ and any morphism $f: R\to \mathbb{P}(\mathcal{E})$, the sheaf $\bar f^*\mathcal{X}^\e$ on $\mathbb{P}^1\times R$ is flat over $R$ with relative degree $d-r$ if and only if $f$ factors through the inclusion $\mathring{Y}_r\embed \mathbb{P}(\mathcal{E})$. In particular, the restriction of $\mathcal{X}^\e$ to $\mathbb{P}^1\times \mathring{Y}_r$ is flat over $\mathring{Y}_r$ with relative degree $d-r$.
\end{prop}

\section{The normal bundles: a first case}          

 Let $V$ be a vector space of dimension $n$ over $\Bbbk$, and $\Gr(k,V)$ be the Grassmannian parametrizing all the $k$-dimensional subspaces of $V$. For any $d\geq 0$, the space $\Mor_d(\mathbb{P}^1,\Gr(k,V))$ of degree $d$ maps from $\mathbb{P}^1$ to $\Gr(k,V)$ is a nonsingular quasi-projective variety. A smooth compactification of $\Mor_d(\mathbb{P}^1,\Gr(k,V))$ is given by the Quot scheme
  $Q_d:=\Quot^{n-k,d}_{V_{\mathbb{P}^1}/\mathbb{P}^1/\Bbbk}$, parametrizing all rank-$(n-k)$, degree-$d$ quotients $V_{\mathbb{P}^1}\surj F$ of the trivial vector bundle $V_{\mathbb{P}^1}$ of rank $n$ on $\mathbb{P}^1$. It comes with a universal exact sequence of sheaves on $\mathbb{P}^1\times Q_d$:
\[
0\to \mathcal{E}_d\to V_{\mathbb{P}^1\times Q_d}\to \mathcal{F}_d\to 0
\]
Here $\mathcal{F}_d$ is flat over $Q_d$ with rank $n-k$ and relative degree $d$. It follows that $\mathcal{E}_d$ is locally free with rank $k$ and relative degree $-d$ over $Q_d$. The open subvariety $\mathring{Q}_d:=\{x\in Q_d\,|\,(\mathcal{F}_d)_x$ is locally free$\}$ coincides with $\Mor_d(\mathbb{P}^1,\Gr(k,V))$.

Recall from the introduction that for any $d>r\geq 0$,  we have the closed subscheme
\[
Z_{d,r}=\{ [V_{\mathbb{P}^1}\surj F] \in Q_d \;|\; \deg(F^\tor) \ge d-r\}.
\]
We refer the reader to Section 3 in \cite{Shao} for the details on the subscheme structure of $Z_{d,r}$.  These are the subschemes that are blown up to yield
the variety $\widetilde Q_d$.  Below we analyze the normal bundle of the locally closed subsets $\mathring{Z}_{d,r}:=Z_{d,r}\setminus Z_{d,r-1}$.

The subscheme $Z_{d,r}$ is closely related to the following relative Quot scheme over $Q_r$:
\[
Q_{d,r}:=\Quot^{0,d-r}_{\mathcal{E}_r/\mathbb{P}^1\times Q_r/Q_r}
\]
If a point of $Q_r$ is represented by the exact sequence $E\inj V_{\mathbb{P}^1}\surj F$, then the fiber of $Q_{d,r}$ over the point consists of points represented by the quotient $E\surj T$ with $T$ torsion of degree $d-r$. Let $\theta_{d,r}: Q_{d,r}\to Q_r$ be the structure morphism. (Note that the notation for this morphism is simply $\theta$ in \cite{Shao}. We add sub-index in this paper because we will deal with multiple $Q_{d,r}$'s together with their structure morphisms simultaneously. The same reason applies to the other similar situations below.) In \cite{Shao}, we showed that $Q_{d,r}$ is relatively smooth over $Q_r$, hence is a nonsingular variety. It comes equipped with a universal exact sequence on $\mathbb{P}^1\times Q_{d,r}$:
\begin{equation}\label{USES-Q(d,r)}
0\to \mathcal{E}_{d,r}\to \theta_{d,r}^*\mathcal{E}_r\to \mathcal{T}_{d,r}\to 0
\end{equation}
Here $\mathcal{T}_{d,r}$ is flat over $Q_{d,r}$ with rank $0$ and relative degree $d-r$. Since $\theta_{d,r}^*\mathcal{E}_r$ is locally free of rank $k$ and of relative degree $-r$ over $Q_{d,r}$, it follows that $\mathcal{E}_{d,r}$ is locally free of rank $k$ and of relative degree $-d$ over $Q_{d,r}$. We set $\mathring{Q}_{d,r}:=\theta_{d,r}^{-1}(\mathring{Q}_r)$ and write $\mathring{\theta}_{d,r}: \mathring{Q}_{d,r}\to Q_r$ for the restriction of $\theta_{d,r}$ on $\mathring{Q}_{d,r}$.

We form a commutative diagram on $\mathbb{P}^1\times Q_{d,r}$:
\begin{equation}\label{bigDiagram}
\CD{
       &                                     & 0\ar[d]                                        & 0\ar[d]                             & \\
0\ar[r]&\mathcal{E}_{d,r}\ar[r]\ar@{=}[d]&\bar\theta_{d,r}^*\mathcal{E}_{r}\ar[r]\ar[d]    & \mathcal{T}_{d,r}\ar[r]\ar@{..>}[d]     & 0 \\
0\ar[r]&\mathcal{E}_{d,r}\ar[r]          &V_{\mathbb{P}^1\times{Q_{d,r}}}\ar[r]\ar[d]    & \mathcal{F}_{d,r}\ar[r]\ar@{..>}[d]     & 0  \\
       &                                     &\bar\theta_{d,r}^*\mathcal{F}_{r}\ar@{=}[r]\ar[d]& \bar\theta_{d,r}^*\mathcal{F}_r\ar[d]& \\
       &                                     &0                                               & 0                                   & \\
}
\end{equation}
where the middle column is the pullback of the universal exact sequence of $Q_r$ via $\bar\theta_{d,r}$, and $\mathcal{F}_{d,r}$ is defined to be the cokernel of the composite map $\mathcal{E}_{d,r}\to \bar\theta_{d,r}^*\mathcal{E}_r\to V_{\mathbb{P}^1\times Q_{d,r}}$. One checks that $\mathcal{F}_{d,r}$ is flat over $Q_{d,r}$ of rank $n-k$ and relative degree $d$. By the universal property of $Q_d$, the middle row determines a morphism $\phi_{d,r}: Q_{d,r}\to Q_d$ (it is denoted by $\phi$ in \cite{Shao}) such that the following diagram commutes:
\begin{equation}\label{DiagramQdr}
\CD{
0\ar[r]&\mathcal{E}_{d,r}\ar[r]\ar@{=}[d]&V_{\mathbb{P}^1\times{Q_{d,r}}}\ar[r]\ar@{=}[d]&\mathcal{F}_{d,r}\ar[r]\ar[d]^{\simeq} & 0  \\
0\ar[r]&\bar\phi_{d,r}^*\mathcal{E}_d\ar[r]  &\bar\phi_{d,r}^*V_{\mathbb{P}^1\times{Q_d}}\ar[r] &\bar\phi_{d,r}^*\mathcal{F}_d\ar[r] &0\\
}
\end{equation}
The morphism $\phi_{d,r}$ maps $Q_{d,r}$ onto $Z_{d,r}$ (cf. Proposition 4.6 of \cite{Shao}).  We
denote by $\mathring{\phi}_{d,r}: \mathring{Q}_{d,r}\to Q_d$ the restriction of $\phi_{d,r}$ on $\mathring{Q}_{d,r}$. Note that $\mathring{\phi}_{d,0}=\phi_{d,0}$ since $\mathring{Q}_{d,0}=Q_{d,0}$. We showed that $\mathring{\phi}_{d,r}$ maps $\mathring{Q}_{d,r}$ into $\mathring{Z}_{d,r}\subset Q_d$ in \cite{Shao}, where $\mathring{Z}_{d,r}= Z_{d,r} \setminus Z_{d,r-1}$, and we set
$Z_{d,-1}:=\varnothing$. We denote by $\mathring{\varphi}_{d,r}: \mathring{Q}_{d,r}\to \mathring{Z}_{d,r}$ the map obtained by restricting the codomain of $\mathring{\phi}_{d,r}$ to $\mathring{Z}_{d,r}$. In \cite{Shao}, we showed that $\mathring{\varphi}_{d,r}: \mathring{Q}_{d,r}\to \mathring{Z}_{d,r}$ is in fact an isomorphism of schemes, hence $\mathring{\phi}_{d,r}: \mathring{Q}_{d,r}\to Q_d$ is an embedding. (cf. Proposition 4.8, \cite{Shao})

\begin{prop}[{\cite{Str87}, Theorem 7.1}]\label{tangentSheaf}
The tangent bundle $\mathcal{T}_{Q_d}$ of $Q_d$ is naturally isomorphic to $\pi_*\HHom(\mathcal{E}_d,\mathcal{F}_d)$. The relative tangent bundle $\mathcal{T}_{Q_{d,r}/Q_r}$ of $Q_{d,r}$ over $Q_r$ is naturally isomorphic to $\pi_*\HHom(\mathcal{E}_{d,r},\mathcal{T}_{d,r})$.
\end{prop}
\begin{proof}
The first assertion is proved in \cite{Str87}, Theorem 7.1. The second, which is a relative version of the first, can be proved by slightly modifying the proof of \cite{Str87}, Theorem 7.1. We omit the details.
\end{proof}

The tangent bundle $\mathcal{T}_{Q_{d,r}}$ of $Q_{d,r}$ fits into the following exact sequence
\begin{equation}\label{relativeTangent}
0\to \mathcal{T}_{Q_{d,r}/Q_r}\to \mathcal{T}_{Q_{d,r}}\to \theta_{d,r}^*\mathcal{T}_{Q_r}\to 0
\end{equation}
The morphism $\mathring{\phi}_{d,r}: \mathring{Q}_{d,r}\to Q_d$ is an embedding and is factored as $\mathring{Q}_{d,0}\stackrel{\mathring{\varphi}_{d,r}}{\stackrel{\sim}{\to}} \mathring{Z}_{d,r}\embed Q_d$ (cf. Proposition 4.8 of \cite{Shao}). Let $\mathcal{N}_{\mathring{Q}_{d,r}/Q_d}$ denote the normal sheaf of $\mathring{Q}_{d,0}$ in $Q_d$. Since $Q_{d,r}$ is nonsingular, $\mathcal{N}_{\mathring{Q}_{d,r}/Q_d}$ is locally free.
\begin{prop}\label{Prop-NormalBundle0}
We have a natural identification $\mathcal{N}_{\mathring{Q}_{d,r}/Q_d}=\pi_*\EExt^1(\mathcal{T}_{d,r},\bar\theta_{d,r}^*\mathcal{F}_r) |_{\mathring{Q}_{d,r}}$. In particular, letting $r=0$, we have
$
\mathcal{N}_{Q_{d,0}/Q_d}= \pi_*\EExt^1(\mathcal{T}_{d,0},\bar\theta_{d,0}^*\mathcal{F}_0).
$
\end{prop}
\begin{proof}
Since $\mathcal{N}_{\mathring{Q}_{d,r}/Q_d}$ is locally free, we have the following exact sequence of sheaves on $\mathring{Q}_{d,r}$:
\[
0\to \mathcal{T}_{\mathring{Q}_{d,r}}\to \mathring{\phi}_{d,r}^*\mathcal{T}_{Q_d} \to \mathcal{N}_{\mathring{Q}_{d,r}/Q_d}\to 0
\]
Note that the restriction of the exact sequence (\ref{relativeTangent}) to $\mathring{Q}_{d,r}$ gives an exact sequence
\[
0\to \mathcal{T}_{\mathring{Q}_{d,r}/\mathring{Q}_r}\to \mathcal{T}_{\mathring{Q}_{d,r}}\to \mathring{\theta}_{d,r}^*\mathcal{T}_{Q_r} \to 0.
\]
Combining the above two sequences, we can form a commutative diagram of sheaves on $\mathring{Q}_{d,r}$:
\[
\CD{
       &                                      & 0\ar[d]                              & 0\ar[d]                             & \\
0\ar[r]&\mathcal{T}_{\mathring{Q}_{d,r}/\mathring{Q}_r}\ar[r]\ar@{=}[d]& \mathcal{T}_{\mathring{Q}_{d,r}}\ar[r]\ar[d] & \mathring{\theta}_{d,r}^*\mathcal{T}_{Q_r}\ar[r]\ar@{..>}[d]     & 0 \\
0\ar[r]&\mathcal{T}_{\mathring{Q}_{d,r}/\mathring{Q}_r}\ar[r]   &\mathring{\phi}_{d,r}^*\mathcal{T}_{Q_d}\ar[r]\ar[d]    & \pi_*\HHom(\mathcal{E}_{d,r},\bar\theta_{d,r}^*\mathcal{F}_r)|_{\mathring{Q}_{d,r}}\ar[r]\ar@{..>}[d]     & 0  \\
       &                                     &\mathcal{N}_{\mathring{Q}_{d,r}/Q_d}\ar@{=}[r]\ar[d]& \mathcal{N}_{\mathring{Q}_{d,r}/Q_d}\ar[d]& \\
       &                                     &0                                               & 0                                   & \\
}
\]
where the middle row comes from the natural identifications
\[
\mathcal{T}_{\mathring{Q}_{d,r}/\mathring{Q}_r} = \pi_*\HHom(\mathcal{E}_{d,r},\mathcal{T}_{d,r}) |_{\mathring{Q}_{d,r}},\quad \mathring{\phi}_{d,r}^*\mathcal{T}_{Q_d} = \mathring{\phi}_{d,r}^*\pi_* \HHom(\mathcal{E}_d,\mathcal{F}_d)=\pi_*\HHom(\mathcal{E}_{d,r},\mathcal{F}_{d,r})|_{\mathring{Q}_{d,r}}
\]
by Proposition \ref{tangentSheaf} and the exact sequence
\[
0\to \pi_*\HHom(\mathcal{E}_{d,r},\mathcal{T}_{d,r})\to \pi_*\HHom(\mathcal{E}_{d,r},\mathcal{F}_{d,r}) \to \pi_*\HHom(\mathcal{E}_{d,r},\bar\theta_{d,r}^*\mathcal{F}_r)\to 0
\]
obtained by applying $\pi_*\HHom(\mathcal{E}_{d,r},-)$ to the third column of the { diagram (\ref{bigDiagram})}.
The dotted arrows in the third column are induced maps on the quotients. Since all rows and the middle column are exact,
the third column is forced to be exact as well.

%
Using the identification
\[
\mathring{\theta}_{d,r}^*\mathcal{T}_{Q_r}= \mathring{\theta}_{d,r}^*\pi_*\HHom(\mathcal{E}_r,\mathcal{F}_r)= \pi_*\HHom(\bar\theta_{d,r}^*\mathcal{E}_r,\bar\theta_{d,r}^*\mathcal{F}_r)|_{\mathring{Q}_{d,r}}
\]
and comparing the third column with the short exact sequence
\[
0\to \pi_*\HHom(\bar\theta_{d,r}^*\mathcal{E}_r,\bar\theta_{d,r}^*\mathcal{F}_r)\to \pi_*\HHom(\mathcal{E}_{d,r},\bar\theta_{d,r}^*\mathcal{F}_r) \to \pi_*\EExt^1(\mathcal{T}_{d,r},\bar\theta_{d,r}^*\mathcal{F}_r)\to 0
\]
obtained by applying $\pi_*\HHom(-,\bar\theta_{d,0}^*\mathcal{F}_0)$ to the exact sequence {(\ref{USES-Q(d,r)})}, we obtain a natural identification
\[
\mathcal{N}_{\mathring{Q}_{d,r}/Q_d}= \pi_*\EExt^1(\mathcal{T}_{d,r},\bar\theta_{d,r}^*\mathcal{F}_r) |_{\mathring{Q}_{d,r}}.
\]
%
%
%
\end{proof}

\section{The normal bundles: the general case}

In this technical section, we introduce and  analyze the properties of a set of auxiliary schemes. These will be used in the final section to identify with
and to derive the desired modular properties of the exceptional divisors created in the sequence of blowups $\widetilde Q_d \to Q_d$.

\subsection{The schemes $Q_{d,r,l}$ and morphisms $ \psi_{d,\ul{r},l}$}
Let $d>r>l\geq 0$, and we consider $Q_{d,r,l}:=Q_{d,r}\times_{Q_r}Q_{r,l}$. Let $\psi_{d,r,\ul{l}}: Q_{d,r,l}\to Q_{d,r}$ and $\psi_{\ul{d},r,l}: Q_{d,r,l}\to Q_{r,l}$ be the two projections. {(Here the underscored subscript $\ul{l}$ in the map $\psi_{d,r,\ul{l}}$ indicates that the subscript $l$ shows up in the source ($Q_{d,r,l}$) but not in the target ($Q_{d,r}$). Ditto for $\ul{d}$, and for $\ul{r}$ below.)} First of all, we see that $Q_{d,r,l}$ is smooth over $Q_{r,l}$ because $Q_{d,r}$ is smooth over $Q_r$. It follows that $Q_{d,r,l}$ is a nonsingular variety. Next, we will define a {finite} morphism $\psi_{d,\ul{r},l}:Q_{d,r,l}\to Q_{d,l}$ based on { the morphisms} in the following diagram
\begin{equation}\label{hexagon}
\CDR{
& Q_r \ar@{}[dd]|-{\textcircled{1}} & \\
Q_{r,l}\ar[ur]^{\phi_{r,l}}\ar@{->>}[dd]_{\theta_{r,l}} & & Q_{d,r}\ar@{->>}[ul]_{\theta_{d,r}}\ar[dd]^{\phi_{d,r}} \\
& Q_{d,r,l}\ar@{->>}[ul]^{\psi_{\ul{d},r,l}}\ar[ur]_{\psi_{d,r,\ul{l}}}\ar@{..>}[dd]^{\psi_{d,\ul{r},l}} & \\
Q_l \ar@{}[ur]|-{\textcircled{2}} & & Q_d \ar@{}[ul]|-{\textcircled{3}} \\
& Q_{d,l}\ar@{->>}[ul]^{\theta_{d,l}}\ar[ur]_{\phi_{d,l}} &
}
\end{equation}
The parallelogram \textcircled{1} is commutative by the definition of $Q_{d,r,l}$. On $\mathbb{P}^r\times Q_{d,r,l}$ we have two short exact sequences:
\[
\begin{split}
& 0\to \bar\psi_{\ul{d},r,l}^*\mathcal{E}_{r,l} \to \bar\psi_{\ul{d},r,l}^*\bar\theta_{r,l}^*\mathcal{E}_l\to \bar\psi_{\ul{d},r,l}^*\mathcal{T}_{r,l}\to 0 \\
& 0\to \bar\psi_{d,r,\ul{l}}^*\mathcal{E}_{d,r} \to \bar\psi_{d,r,\ul{l}}^*\bar\theta_{d,r}^*\mathcal{E}_r\to \bar\psi_{d,r,\ul{l}}^*\mathcal{T}_{d,r}\to 0
\end{split}
\]
which are pullbacks of the universal exact sequences of $Q_{r,l}$ and $Q_{d,r}$ via $\bar\psi_{\ul{d},r,l}$ and $\bar\psi_{d,r,\ul{l}}$ respectively. Note that $\bar\psi_{d,r,\ul{l}}^*\bar\theta_{d,r}^*\mathcal{E}_r=\bar\psi_{\ul{d},r,l}^*\bar\phi_{r,l}^*\mathcal{E}_r=\bar\psi_{\ul{d},r,l}\mathcal{E}_{r,l}$. Putting the two sequences together, we can form a commutative diagram as follows:
\begin{equation}\label{Diagram-Q(d,r,l)}
\CD{
       &  & 0\ar[d]  & 0\ar[d]  & \\
0\ar[r]&\bar\psi_{d,r,\ul{l}}^*\mathcal{E}_{d,r}\ar[r]\ar@{=}[d]& \bar\psi_{\ul{d},r,l}^*\mathcal{E}_{r,l}\ar[r]\ar[d] & \bar\psi_{d,r,\ul{l}}^*\mathcal{T}_{d,r}\ar[r]\ar@{..>}[d]     & 0 \\
0\ar[r]&\bar\psi_{d,r,\ul{l}}^*\mathcal{E}_{d,r}\ar[r]   &\bar\psi_{\ul{d},r,l}^*\bar\theta_{r,l}^*\mathcal{E}_l\ar[r]\ar[d]    & \mathcal{T}\ar[r]\ar@{..>}[d]     & 0  \\
       &  &\bar\psi_{\ul{d},r,l}^*\mathcal{T}_{r,l}\ar@{=}[r]\ar[d]& \bar\psi_{\ul{d},r,l}^*\mathcal{T}_{r,l}\ar[d]& \\
       &  &0  & 0  & \\
}
\end{equation}
where $\mathcal{T}=\Coker(\bar\psi_{d,r,\ul{l}}^*\mathcal{E}_{d,r}\to\bar\psi_{\ul{d},r,l}^*\mathcal{E}_{r,l} \to\bar\psi_{\ul{d},r,l}^*\bar\theta_{r,l}^*\mathcal{E}_l)$, and the dotted arrows are the induced maps on quotients. Since the upper two rows and the middle column are exact, the last column are forced to be exact as well. Note that $\bar\psi_{d,r,\ul{l}}^*\mathcal{T}_{d,r}$ and $\bar\psi_{\ul{d},r,l}^*\mathcal{T}_{r,l}$ are both flat over $Q_{d,r,l}$ with rank 0 but with relative degree $d-r$ and $r-l$ respectively. Hence $\mathcal{T}$ is flat over $Q_{d,r,l}$ as well and is of rank 0 and of relative degree $d-l$. Thus the quotient $\bar\psi_{\ul{d},r,l}^*\bar\theta_{r,l}^*\mathcal{E}_l\surj \mathcal{T}$ from the middle row determines a $Q_l$-morphism
\[
\psi_{d,\ul{r},l}: Q_{d,r,l}\to Q_{d,l}
\]
such that the pullback of the universal exact sequence of $Q_{d,l}$ is the middle row of (\ref{Diagram-Q(d,r,l)}), i.e., we have the following identifications:
\begin{equation}\label{Identification-psi}
\CD{
0\ar[r]&\bar\psi_{d,r,\ul{l}}^*\mathcal{E}_{d,r}\ar[r]\ar@{=}[d]   &\bar\psi_{\ul{d},r,l}^*\bar\theta_{r,l}^*\mathcal{E}_l\ar[r]\ar@{=}[d]    & \mathcal{T}\ar[r]\ar@{=}[d]     & 0 \\
 0\ar[r] & \bar\psi_{d,\ul{r},l}^*\mathcal{E}_{d,l}\ar[r] & \bar\psi_{d,\ul{r},l}^*\bar\theta_{d,l}^*\mathcal{E}_l\ar[r] & \bar\psi_{d,\ul{r},l}^*\mathcal{T}_{d,l}\ar[r] & 0 \\
}
\end{equation}
In particular, the identification $\mathcal{T}=\bar\psi_{d,\ul{r},l}^*\mathcal{T}_{d,l}$ allows us to rewrite the third column of the diagram (\ref{Diagram-Q(d,r,l)}) as
\begin{equation}\label{SES-Torsions}
0\to\bar\psi_{d,r,\ul{l}}^*\mathcal{T}_{d,r}\to \bar\psi_{d,\ul{r},l}^*\mathcal{T}_{d,l}\to \bar\psi_{\ul{d},r,l}^*\mathcal{T}_{r,l}\to 0
\end{equation}
By the definition of $\psi_{d,\ul{r},l}$, { the parallelogram \textcircled{2}} in the diagram (\ref{hexagon}) automatically commutes. We also have
\begin{prop}
The parallelogram \textcircled{3} in diagram (\ref{hexagon}) commutes.
\end{prop}
\begin{proof}
We need to show $\phi_{d,r}\psi_{d,r,\ul{l}}=\phi_{d,l}\psi_{d,\ul{r},l}$. Because both $\phi_{d,r}\psi_{d,r,\ul{l}}$ and $\phi_{d,l}\psi_{d,\ul{r},l}$ map into $Q_d$, by the universal property of the Quot scheme $Q_d$ (see \cite{Shao}, Theorem 2.1), it suffices to show that there is an isomorphism $(\overline{\phi_{d,r}\psi_{d,r,\ul{l}}})^*\mathcal{F}_d\simeq (\overline{\phi_{d,l}\psi_{d,\ul{r},l}})^*\mathcal{F}_d$ that makes the following diagram commute:
\[
\CD{
V_{\mathbb{P}^1\times Q_{d,r,l}}\ar@{=}[r]\ar@{=}[d] & (\overline{\phi_{d,r}\psi_{d,r,\ul{l}}})^* V_{\mathbb{P}^1\times Q_d}\ar[r] & (\overline{\phi_{d,r}\psi_{d,r,\ul{l}}})^*\mathcal{F}_d\ar[d]^{\simeq} \\
V_{\mathbb{P}^1\times Q_{d,r,l}}\ar@{=}[r] & (\overline{\phi_{d,l}\psi_{d,\ul{r},l}})^* V_{\mathbb{P}^1\times Q_d}\ar[r] & (\overline{\phi_{d,l}\psi_{d,\ul{r},l}})^*\mathcal{F}_d
}
\]
By the diagram (\ref{DiagramQdr}), we have an exact sequence
\[
0\to \mathcal{E}_{d,r}\to V_{\mathbb{P}^1\times Q_{d,r}}\to \bar\phi_{d,r}^*\mathcal{F}_d\to 0
\]
Replacing $r$ with $l$, we obtain another one
\[
0\to \mathcal{E}_{d,l}\to V_{\mathbb{P}^1\times Q_{d,l}}\to \bar\phi_{d,l}^*\mathcal{F}_d\to 0
\]
Applying $\bar\psi_{d,r,\ul{l}}^*$ and $\bar\psi_{d,\ul{r},l}^*$ to the above two exact sequences respectively, we obtain two exact sequences
\[
\begin{split}
& 0\to \bar\psi_{d,r,\ul{l}}^*\mathcal{E}_{d,r}\to V_{\mathbb{P}^1\times Q_{d,r,l}}\to \bar\psi_{d,r,\ul{l}}^*\bar\phi_{d,r}^*\mathcal{F}_d\to 0 \\
& 0\to \bar\psi_{d,\ul{r},l}^*\mathcal{E}_{d,l}\to V_{\mathbb{P}^1\times Q_{d,r,l}}\to \bar\psi_{d,\ul{r},l}^*\bar\phi_{d,l}^*\mathcal{F}_d\to 0
\end{split}
\]
Using the identification $\bar\psi_{d,r,\ul{l}}^*\mathcal{E}_{d,r}=\bar\psi_{d,\ul{r},l}^*\mathcal{E}_{d,l}$ from the diagram (\ref{Identification-psi}), we obtain a commutative diagram
\[
\CD{
0\ar[r] & \bar\psi_{d,r,\ul{l}}^*\mathcal{E}_{d,r}\ar[r]\ar@{=}[d] & V_{\mathbb{P}^1\times Q_{d,r,l}}\ar[r]\ar@{=}[d] & \bar\psi_{d,r,\ul{l}}^*\bar\phi_{d,r}^*\mathcal{F}_d\ar[r]\ar@{..>}[d]^{\simeq} & 0 \\
0\ar[r] & \bar\psi_{d,\ul{r},l}^*\mathcal{E}_{d,l}\ar[r] & V_{\mathbb{P}^1\times Q_{d,r,l}}\ar[r] & \bar\psi_{d,\ul{r},l}^*\bar\phi_{d,l}^*\mathcal{F}_d\ar[r] & 0
}
\]
which induces an isomorphism $\bar\psi_{d,r,\ul{l}}^*\bar\phi_{d,r}^*\mathcal{F}_d\simeq \bar\psi_{d,\ul{r},l}^*\bar\phi_{d,l}^*\mathcal{F}_d$ in the third column. This isomorphism is the desired one.
\end{proof}

The map $\psi_{d,\ul{r},l}$ is both proper and quasi-finite, hence it is a finite morphism. 

\begin{notation}
For any integer $m\geq 1$, let $\Sigma_m$ be the set of all strictly decreasing sequences of nonnegative integers of length $m$:
\[
\Sigma_m=\{(r_1,\cdots,r_m)\,|\,r_1>\cdots> r_m\}
\]
and let $\Sigma$ be the set of all strictly decreasing sequences of nonnegative integers of any finite length:
\[
\Sigma=\bigcup_{m=1}^\infty\Sigma_m
\]
For any sequence $\sigma=(r_1,\cdots,r_m)\in \Sigma$, the first term $r_1$ is called the leading term and is denoted as $\lt(\sigma)$. For any integer $r>\lt(\sigma)$, by $(r,\sigma)$ we mean the new sequence $(r,r_1,\cdots,r_m)$. The length of a sequence $\sigma$ is denoted by $|\sigma|$.
\end{notation}
{When we use a sequence of $\Sigma$ as sub-index, we would omit the parentheses. For example, if $\sigma=(8,5,3,1,0)$, $\tau=(5,3,1,0)$ and $\lambda=(3,1,0)$, then the notations $Q_\sigma$, $Q_{8,\tau}$ and $Q_{8,5,\lambda}$ all mean the same thing: $Q_{8,5,3,1,0}$.}

\subsection{The schemes $P_\sigma$ and their properties}
We { now introduce a set of spaces $P_\sigma$ together with a set of coherent sheaves $\mathcal{X}_\sigma$ on $\mathbb{P}^1\times P_\sigma$, indexed by} $\sigma\in\Sigma$. We will need a set of auxiliary spaces $R_\sigma$ indexed by $\sigma\in\Sigma$ with $|\sigma|\geq 2$. First, we define $P_\sigma$ for $\sigma\in \Sigma_1$. Suppose $\sigma=(d)$. In this case, we set
\[
P_\sigma=P_d:=Q_d,\quad \mathring{P}_\sigma=\mathring{P}_d:=\mathring{Q}_d,\quad \mathcal{X}_\sigma=\mathcal{X}_d:=\mathcal{F}_d
\]
We will denote by $\mathring{\mathcal{X}}_d$ the restriction of $\mathcal{X}_d$ on $\mathbb{P}^1\times\mathring{P}_d$. We know that $\mathring{\mathcal{X}}_d$ is locally free.

Next we define $P_\sigma$ for $\sigma\in \Sigma_2$. Suppose $\sigma=(d,r)$, $d>r$. We set
\[
R_{d,r}:=Q_{d,r}\times_{Q_r}P_r,\quad \mathring{R}_{d,r}:=Q_{d,r}\times_{Q_r}\mathring{P}_r\subset R_{d,r}
\]
Since $P_r=Q_r$ and $\mathring{P}_r=\mathring{Q}_r$, we actually have $R_{d,r}=Q_{d,r}$ and $\mathring{R}_{d,r}=\mathring{Q}_{d,r}$. Let $q_{d,r}$ and $p_{d,r}$ be the projections from $R_{d,r}$ to $Q_{d,r}$ and to $P_r$, respectively, and let $\mathring{q}_{d,r}$ and $\mathring{p}_{d,r}$ be the projections from $\mathring{R}_{d,r}$ to $Q_{d,r}$ and to $\mathring{P}_r$, respectively. We define $P_{d,r}$ to be a (relative) space of non-split extensions:
\[
P_{d,r}=\mathbb{P}(\pi_*\EExt^1({\bar{\mathring{q}}}_{d,r}^*\mathcal{T}_{d,r},{\bar{\mathring{p}}}_{d,r}^*\mathring{\mathcal{X}}_r))
\]
and let $a_{d,r}: P_{d,r}\to \mathring{R}_{d,r}$ be the structure morphism. We denote the universal extension on $\mathbb{P}^1\times P_{d,r}$ by
\[
0\to (\bar a_{d,r}^*{\bar{\mathring{p}}}_{d,r}^*\mathring{\mathcal{X}}_r)(0,1)\to \mathcal{X}_{d,r}\to \bar a_{d,r}^*{\bar {\mathring{q}}}_{d,r}^*\mathcal{T}_{d,r}\to 0
\]
Let $\mathring{P}_{d,r}\subset P_{d,r}$ be the open subset defined as
\[
\mathring{P}_{d,r}:=\{x\in P_{d,r}\,|\,(\mathcal{X}_{d,r})_x\text{ is locally free}\}
\]
We denote by $\mathring{\mathcal{X}}_{d,r}$ the restriction of $\mathcal{X}_{d,r}$ to $\mathbb{P}^1\times \mathring{P}_{d,r}$. Then $\mathring{\mathcal{X}}_{d,r}$ is locally free, and $P_{d,r}$ can be considered as a $Q_d$-scheme through the composition $P_{d,r}\to \mathring{R}_{d,r}\to Q_{d,r}\to Q_d$. So we have defined $P_\sigma$, $\mathcal{X}_\sigma$, $R_\sigma$, etc., for any $\sigma\in \Sigma$ with $|\sigma|=2$. In the following, we will define $P_\sigma$, $\mathcal{X}_\sigma$, etc., for any $\sigma\in\Sigma$ with $|\sigma|\geq 3$ inductively.



Assume that, for each $\sigma\in\Sigma_m$ for some $m\geq 2$, the space $P_\sigma$ of non-split extensions is defined and the sheaf $\mathcal{X}_\sigma$ is the middle term from the universal extension on $\mathbb{P}^1\times P_\sigma$ . Assume also that a morphism $P_\sigma\to Q_{l}$ ($l=\lt(\sigma)$) has been specified so that $P_\sigma$ can be considered as a $Q_l$-scheme.

Let $\sigma\in\Sigma_{m+1}$, $d=\lt(\sigma)$, $\tau$ be the sequence formed from $\sigma$ by removing the leading term $d$, and $r=\lt(\tau)$. So $\sigma=(d,\tau)=(d,r,\cdots)$. By induction hypothesis, the space $P_\tau$ of non-split extensions is defined and is a $Q_r$-scheme. We set
\[
R_\sigma:=Q_{d,r}\times_{Q_r}P_\tau,\quad \mathring{R}_\sigma:=Q_{d,r}\times_{Q_r}\mathring{P}_\tau\subset R_\sigma
\]
Let $q_\sigma$ and $p_\sigma$ be the two projections from $R_\sigma$ to $Q_{d,r}$ and to $P_\tau$, respectively, and let $\mathring q_\sigma$ and $\mathring p_\sigma$ be the two projections from $\mathring{R}_\sigma$ to $Q_{d,r}$ and to $\mathring{P}_\tau$, respectively. We define $P_\sigma=P_{d,\tau}$ to be a space of non-split extensions over $\mathring{R}_\sigma$ by
\[
P_\sigma:=\mathbb{P}(\pi_*\EExt^1({\bar{\mathring{q}}}_\sigma^*\mathcal{T}_{d,r},{\bar{\mathring{p}}}_\sigma^*\mathring{\mathcal{X}}_\tau))
\]
and let $a_\sigma: P_\sigma\to \mathring{R}_\sigma$ be the structure morphism. We denote the universal extension on $\mathbb{P}^1\times P_\sigma$ by
\[
0\to (\bar a_\sigma^*{\bar{\mathring p}}_\sigma^*\mathring{\mathcal{X}}_\tau)(0,1)\to \mathcal{X}_\sigma\to \bar a_\sigma^*{\bar{\mathring q}}_\sigma^*\mathcal{T}_{d,r}\to 0
\]
$P_\sigma$ can be considered as a $Q_d$-scheme through $P_\sigma\to \mathring{R}_\sigma\to Q_{d,r}\to Q_d$. We define the open subset $\mathring{P}_\sigma\subset P_\sigma$ as
\[
\mathring{P}_\sigma:=\{x\in P_\sigma\,|\,(\mathcal{X}_\sigma)_x\text{ is locally free}\}
\]
Then, $\mathring{\mathcal{X}}_\sigma$, the restriction of $\mathcal{X}_\sigma$ on $\mathbb{P}^1\times \mathring{P}_\sigma$ is locally free. By induction, we have defined $P_\sigma$, $\mathcal{X}_\sigma$, $R_\sigma$, etc., for all $\sigma\in\Sigma$.

\begin{lemma}
For each $\sigma$,
\begin{enumerate}
\item $\mathcal{X}_\sigma$ is flat over $\mathbb{P}_\sigma$ with relative degree $\lt(\sigma)$.
\item $R^1\pi_*(\mathcal{X}_\sigma(-1))=0$.
\end{enumerate}
\end{lemma}

The closed subset $P_\sigma\setminus \mathring{P}_\sigma$ of $P_\sigma$ has a sequence of nested closed subschemes
\[
\varnothing=Y_{d,r,\tau}\subset Y_{d,r+1,\tau}\subset\cdots\subset Y_{d,d-1,\tau}=P_\sigma\setminus \mathring{P}_\sigma
\]
where $Y_{d,e,\tau}=\{x\in P_\sigma\,|\,\deg((\mathcal{X}_\sigma)_x^\tor)\geq d-e \}$ and the subscheme structure on $Y_{d,e,\tau}$ is defined by (\ref{subscheme}). We set $\mathring{Y}_{d,e,\tau}:=Y_{d,e,\tau}\setminus Y_{d,e-1,\tau}$ for $e=r+1,\cdots,d-1$. That is, $\mathring{Y}_{d,e,\tau}=\{x\in P_\sigma\,|\,\deg((\mathcal{X}_\sigma)_x^\tor)= d-e \}$.

\begin{lemma}
The space $P_\sigma$ is nonsingular for any $\sigma\in\Sigma$.
\end{lemma}
\begin{proof}
{We prove it by induction on $|\sigma|$. When $|\sigma|=1$, say $\sigma=(d)$ for some $d$, we have $P_\sigma=P_d=Q_d$. So this case is obvious since $Q_d$ is nonsingular. Assume that the statement holds true for all $\sigma\in \Sigma_m$ for some $m$. Let $\sigma\in\Sigma_{m+1}$ and suppose $\sigma=(d,\tau)$ where $\tau\in\Sigma_m$. By definition, $P_\sigma=P_{d,\tau}$ is a projective bundle over $Q_{d,r}\times_{Q_r}\mathring{P}_\tau$ where $r=\lt(\tau)$. Hence $P_\sigma$ is smooth over $Q_{d,r}\times_{Q_r}\mathring{P}_\tau$. We know that $Q_{d,r}$ is smooth over $Q_r$, hence $Q_{d,r}\times_{Q_r}\mathring{P}_\tau$ is smooth over $\mathring{P}_\tau$. Since smoothness is transitive, $P_\sigma$ is smooth over $\mathring{P}_\tau$. By induction hypothesis, $P_\tau$ is nonsingular, and so is $\mathring{P}_\tau$ and hence $P_\sigma$ is nonsingular. This completes the proof.}
\end{proof}

Let $\sigma\in\Sigma$, $l=\lt(\sigma)$ and let $d>r>l$. We define a morphism $\phi_{d,r,\sigma}: R_{d,r,\sigma}\to P_{d,\sigma}$ using the following commutative diagram
\begin{equation}
\xymatrix@C=4pc{
P_{r,\sigma}\ar[d]_{a_{r,\sigma}} & Q_{d,r}\times_{Q_r}P_{r,\sigma}=R_{d,r,\sigma}\ar[l]_-{p_{d,r,\sigma}}\ar@{.>}[r]^-{\phi_{d,r,\sigma}}\ar[d]_{1\times a_{r,\sigma}} & P_{d,\sigma}\ar[d]^{a_{d,\sigma}} \\
Q_{r,l}\times_{Q_l}\mathring{P}_\sigma \ar[d]_{\mathring{q}_{r,\sigma}} & Q_{d,r}\times_{Q_r}Q_{r,l}\times_{Q_l}\mathring{P}_\sigma \ar[l]_-{\psi_{\ul{d},r,l}\times 1}\ar[r]^-{\psi_{d,\ul{r},l}\times 1}\ar[d]_{1\times\mathring{q}_{r,\sigma}} & Q_{d,l}\times_{Q_l}\mathring{P}_\sigma \ar[d]^{\mathring{q}_{d,\sigma}} \\
Q_{r,l} & Q_{d,r}\times_{Q_r}Q_{r,l}\ar[l]_-{\psi_{\ul{d},r,l}}\ar[r]^-{\psi_{d,\ul{r},l}} & Q_{d,l}
}
\end{equation}
By the base change property, $R_{d,r,\sigma}$ is a projective bundle over $Q_{d,r}\times_{Q_r}\mathring{R}_{r,\sigma}=Q_{d,r}\times_{Q_r}Q_{r,l}\times_{Q_l}\mathring{P}_\sigma$
\[
R_{d,r,\sigma}=(Q_{d,r}\times_{Q_r}\mathring{R}_{r,\sigma})\times_{\mathring{R}_{r,\sigma}}P_{r,\sigma} =\mathbb{P}((\psi_{\ul{d},r,l}\times 1)^*\pi_*\EExt^1({\bar{\mathring{q}}}_{r,\sigma}^*\mathcal{T}_{r,l},{\bar{\mathring{p}}}_{r,\sigma}^*\mathring{\mathcal{X}}_\sigma))
\]
with $\mathcal{O}_{R_{d,r,\sigma}}(1)=p_{d,r,\sigma}^*\mathcal{O}_{P_{r,\sigma}}(1)$ and structure morphism $1\times a_{r,\sigma}$. On $\mathbb{P}^1\times R_{d,r,\sigma}$, we have a commutative diagram of sheaves
\begin{equation}\label{Diagram-R(d,r,l)}
\CD{
 & & 0\ar[d] & 0\ar[d] & \\
 & & \bar q_{d,r,\sigma}^*\mathcal{T}_{d,r}\ar@{=}[r]\ar[d] & \overline{`a_{r,\sigma}}\,^*\overline{`\mathring q_{r,\sigma}}\,^*\bar \psi_{d,r,\ul{l}}^*\mathcal{T}_{d,r}\ar[d] \\
0\ar[r] & (\overline{`a_{r,\sigma}}\,^*\overline{\psi_{d,\ul{r},l}'}\,^*{\bar{\mathring p}}_{d,\sigma}^*\mathring{\mathcal{X}}_\sigma)(0,1) \ar[r]\ar@{=}[d] & \mathcal{X}\ar[r]\ar[d] & \overline{` a_{r,\sigma}}\,^*\overline{\psi_{d,\ul{r},l}'}\,^*{\bar{\mathring q}}_{d,\sigma}^*\mathcal{T}_{d,l}\ar[r]\ar[d] & 0 \\
0\ar[r] & \bar p_{d,r,\sigma}^*((\bar a_{r,\sigma}^*{\bar{\mathring p}}_{r,\sigma}^*\mathring{\mathcal{X}}_\sigma)(0,1))\ar[r] &\bar p_{d,r,\sigma}^*\mathcal{X}_{r,\sigma}\ar[r]\ar[d] & \bar p_{d,r,\sigma}^*\bar a_{r,\sigma}^*{\bar{\mathring{q}}}_{r,\sigma}^*\mathcal{T}_{r,l}\ar[r]\ar[d] & 0 \\
 & & 0 & 0
}
\end{equation}
where
\begin{itemize}
\item {for short notations, we have set $`a_{r,\sigma}:=1\times a_{r,\sigma}$, $`\mathring q_{r,\sigma}:=1\times q_{r,\sigma}$, $\psi_{d,\ul{r},l}':=\psi_{d,\ul{r},l}\times 1$,}
\item the last row is the pullback of the universal extension of $P_{r,\sigma}$ via $\bar p_{d,r,\sigma}: \mathbb{P}^1\times R_{d,r,\sigma}\to \mathbb{P}^1\times P_{r,\sigma}$,
\item the last column is the pullback of the exact sequence of torsion sheaves via $\overline{`a_{r,\sigma}}\,\overline{` \mathring q_{r,\sigma}}: \mathbb{P}^1\times R_{d,r,\sigma}\to \mathbb{P}^1\times(Q_{d,r}\times_{Q_r}Q_{r,l})$, and
\item $\mathcal{X}$ is the fiber product $(\bar p_{d,r,\sigma}^*\mathcal{X}_{r,\sigma})\times_{(\bar p_{d,r,\sigma}^*\bar a_{r,\sigma}^*{\bar{\mathring{q}}}_{r,\sigma}^*\mathcal{T}_{r,l})}(\overline{` a_{r,\sigma}}\,^*\overline{\psi_{d,\ul{r},l}'}\,^*{\bar{\mathring q}}_{r,\sigma}^*\mathcal{T}_{d,l})$ in the category of coherent sheaves.
\end{itemize}
Recall that the universal extension of $P_{d,\sigma}$ is the exact sequence
\begin{equation}\label{SES-P(d,sigma)}
0\to (\bar a_{d,\sigma}^*{\bar{\mathring p}}_{d,\sigma}^*\mathring{\mathcal{X}}_\sigma)(0,1)\to \mathcal{X}_{d,\sigma}\to \bar a_{d,\sigma}^*{\bar{\mathring q}}_{d,\sigma}^*\mathcal{T}_{d,l}\to 0
\end{equation}
on $\mathbb{P}^1\times P_{d,\sigma}$. By Theorem \ref{thm-universal}, the middle row determines a $(Q_{d,l}\times_{Q_l}\mathring{P}_\sigma)$-morphism
\[
\phi_{d,r,\sigma}: R_{d,r,\sigma}\to P_{d,\sigma}
\]
such that there are isomorphisms $\mathcal{O}_{R_{d,r,\sigma}}(1)\simeq \phi_{d,r,\sigma}^*\mathcal{O}_{P_{d,\sigma}}(1)$ and $\mathcal{X}\simeq \bar\phi_{d,r,\sigma}^*\mathcal{X}_{d,\sigma}$ that make the following diagram commute:
\[
\CD{
0\ar[r] & (\overline{`a_{r,\sigma}}\,^*\overline{\psi_{d,\ul{r},l}'}\,^*{\bar{\mathring p}}_{d,\sigma}^*\mathring{\mathcal{X}}_\sigma)(0,1) \ar[r]\ar[d]^{\simeq} & \mathcal{X}\ar[r]\ar[d]^{\simeq} & \overline{` a_{r,\sigma}}\,^*\overline{\psi_{d,\ul{r},l}'}\,^*{\bar{\mathring q}}_{d,\sigma}^*\mathcal{T}_{d,l}\ar[r]\ar@{=}[d] & 0 \\
0\ar[r] & \bar\phi_{d,r,\sigma}^*((\bar a_{d,\sigma}^*{\bar{\mathring p}}_{d,\sigma}^*\mathring{\mathcal{X}}_\sigma)(0,1))\ar[r] & \bar\phi_{d,r,\sigma}^*\mathcal{X}_{d,\sigma}\ar[r] & \bar\phi_{d,r,\sigma}^*\bar a_{d,\sigma}^*{\bar{\mathring q}}_{d,\sigma}^*\mathcal{T}_{d,l}\ar[r] & 0
}
\]
where the first row is the middle row of diagram (\ref{Diagram-R(d,r,l)}) and the second row is the pullback of the universal extension (\ref{SES-P(d,sigma)}) of $P_{d,\sigma}$ via $\bar\phi_{d,r,\sigma}: \mathbb{P}^1\times R_{d,r,\sigma}\to \mathbb{P}^1\times P_{d,\sigma}$.
For simplicity, we make identifications
\begin{equation}\label{Identification-phi}
\mathcal{O}_{R_{d,r,\sigma}}(1)=\phi_{d,r,\sigma}^*\mathcal{O}_{P_{d,\sigma}}(1),\qquad \mathcal{X}=\bar\phi_{d,r,\sigma}^*\mathcal{X}_{d,\sigma}.
\end{equation}

Let $\mathring{\phi}_{d,r,\sigma}: \mathring{R}_{d,r,\sigma}\to P_{d,\sigma}$ be the restriction of $\phi_{d,r,\sigma}$ to $\mathring{R}_{d,r,\sigma}$. The restriction of the middle column of diagram (\ref{Diagram-R(d,r,l)}) to $\mathbb{P}^1\times\mathring{R}_{d,r,\sigma}$ is the exact sequence
\begin{equation}\label{SES-R(d,r,l)}
0\to {\bar{\mathring q}}_{d,r,\sigma}^*\mathcal{T}_{d,r}\to \bar{\mathring{\phi}}_{d,r,\sigma}^*\mathcal{X}_{d,\sigma}\to {\bar{\mathring p}}_{d,r,\sigma}^*\mathring{\mathcal{X}}_{r,\sigma}\to 0
\end{equation}
\begin{prop}
The map $\mathring{\phi}_{d,r,\sigma}$ factors through the inclusion $\mathring{Y}_{d,r,\sigma}\embed P_{d,\sigma}$.
\end{prop}
\begin{proof}
Taking the dual of the sequence (\ref{SES-R(d,r,l)}), we obtain $(\bar{\mathring{\phi}}_{d,r,\sigma}^*\mathcal{X}_{d,\sigma})^\e \simeq ({\bar{\mathring q}}_{d,r,\sigma}^*\mathcal{T}_{d,r})^\e$. Since $(\bar{\mathring{\phi}}_{d,r,\sigma}^*\mathcal{X}_{d,\sigma})^\e=\bar{\mathring{\phi}}_{d,r,\sigma}^*(\mathcal{X}_{d,\sigma}^\e)$, we have that $\bar{\mathring{\phi}}_{d,r,\sigma}^*(\mathcal{X}_{d,\sigma}^\e)$ is flat over $\mathring{R}_{d,r,\sigma}$ with relative degree $d-r$. By Proposition \ref{FlatteningStratification}, the map $\mathring{\phi}_{d,r,\sigma}$ factors through the inclusion $\mathring{Y}_{d,r,\sigma}\embed P_{d,\sigma}$.
\end{proof}

We denote by $\mathring{\varphi}_{d,r,\sigma}: \mathring{R}_{d,r,\sigma}\to \mathring{Y}_{d,r,\sigma}$ the map factored out from $\mathring{\phi}_{d,r,\sigma}$.
\begin{prop}
The morphism $\mathring{\varphi}_{d,r,\sigma}: \mathring{R}_{d,r,\sigma}\to \mathring{Y}_{d,r,\sigma}$ is an isomorphism.
\end{prop}
\begin{proof}
We prove by constructing an inverse of $\mathring{\varphi}_{d,r,\sigma}$. Let $i: \mathring{Y}_{d,r,\sigma}\embed P_{d,\sigma}$ be the inclusion map. The pullback of the universal extension (\ref{SES-P(d,sigma)}) via $i$ is the exact sequence
\[
0\to \bar i^*((\bar a_{d,\sigma}^*{\bar{\mathring p}}_{d,\sigma}^*\mathring{\mathcal{X}}_\sigma)(0,1))\to \bar i^*\mathcal{X}_{d,\sigma}\to \bar i^*\bar a_{d,\sigma}^*{\bar{\mathring q}}_{d,\sigma}^*\mathcal{T}_{d,l}\to 0
\]
on $\mathbb{P}^1\times \mathring{Y}_{d,r,\sigma}$. Taking dual, we obtain a long exact sequence
\[
0\to (\bar i^*\mathcal{X}_{d,\sigma})^\vee \to (\bar i^*((\bar a_{d,\sigma}^*{\bar{\mathring p}}_{d,\sigma}^*\mathring{\mathcal{X}}_\sigma)(0,1)))^\vee \to (\bar i^*\bar a_{d,\sigma}^*{\bar{\mathring q}}_{d,\sigma}^*\mathcal{T}_{d,l})^\e \to (\bar i^*\mathcal{X}_{d,\sigma})^\e \to 0
\]
We break it into two short exact sequences
\[
0\to (\bar i^*\mathcal{X}_{d,\sigma})^\vee \to (\bar i^*((\bar a_{d,\sigma}^*{\bar{\mathring p}}_{d,\sigma}^*\mathring{\mathcal{X}}_\sigma)(0,1)))^\vee \to \mathcal{T} \to 0
\]
\[
0\to \mathcal{T} \to (\bar i^*\bar a_{d,\sigma}^*{\bar{\mathring q}}_{d,\sigma}^*\mathcal{T}_{d,l})^\e \to (\bar i^*\mathcal{X}_{d,\sigma})^\e \to 0
\]
We have $(\bar i^*\mathcal{X}_{d,\sigma})^\e=\bar i^*(\mathcal{X}_{d,\sigma}^\e)$ and $(\bar i^*\bar a_{d,\sigma}^*{\bar{\mathring q}}_{d,\sigma}^*\mathcal{T}_{d,l})^\e=\bar i^*\bar a_{d,\sigma}^*{\bar{\mathring q}}_{d,\sigma}^*(\mathcal{T}_{d,l}^\e)$. By Proposition \ref{FlatteningStratification}, $\bar i^*(\mathcal{X}_{d,\sigma}^\e)$ is flat over $\mathring{Y}_{d,r,\sigma}$ with relative degree $d-r$. Since $\mathcal{T}_{d,l}^\e$ is also flat over $Q_{d,l}$ with relative degree $d-l$, we have $\mathcal{T}$ is flat over $\mathring{Y}_{d,r,\sigma}$ with relative degree $r-l$. We also know that both $(\bar i^*\bar a_{d,\sigma}^*{\bar{\mathring q}}_{d,\sigma}^*\mathcal{T}_{d,l})^\e$ and $(\bar i^*\mathcal{X}_{d,\sigma})^\e$ are torsion, hence $\mathcal{T}$ is also torsion. Since the middle term of the first sequence is locally free and the last term is flat over $\mathring{Y}_{d,r,\sigma}$, the first term, $(\bar i^*\mathcal{X}_{d,\sigma})^\vee$, is locally free as well.

Now dualizing both of the above sequences, we obtain another two exact sequences
\begin{equation}\label{sequence1}
0\to  \bar i^*((\bar a_{d,\sigma}^*{\bar{\mathring p}}_{d,\sigma}^*\mathring{\mathcal{X}}_\sigma)(0,1)) \to (\bar i^*\mathcal{X}_{d,\sigma})^{\vee\vee} \to \mathcal{T}^\e \to 0
\end{equation}
\begin{equation}\label{sequence2}
0\to (\bar i^*\mathcal{X}_{d,\sigma})^{\e\e}  \to \bar i^*\bar a_{d,\sigma}^*{\bar{\mathring q}}_{d,\sigma}^*\mathcal{T}_{d,l} \to \mathcal{T}^\e \to 0
\end{equation}
The Quot scheme $Q_{d,l}$ has a universal quotient $\bar\theta_{d,l}^*\mathcal{E}_l\surj \mathcal{T}_{d,l}$. Applying $\bar i^*\bar a_{d,\sigma}^*{\bar{\mathring q}}_{d,\sigma}^*$ to this quotient, we obtain a quotient $\bar i^*\bar a_{d,\sigma}^*{\bar{\mathring q}}_{d,\sigma}^*\bar\theta_{d,l}^*\mathcal{E}_l\surj \bar i^*\bar a_{d,\sigma}^*{\bar{\mathring q}}_{d,\sigma}^*\mathcal{T}_{d,l}$. Its composition with the quotient $\bar i^*\bar a_{d,\sigma}^*{\bar{\mathring q}}_{d,\sigma}^*\mathcal{T}_{d,l} \surj \mathcal{T}^\e$ from the sequence (\ref{sequence2}) yields a quotient

We form a commutative diagram
\[
\CD{
& & 0\ar[d] & 0\ar[d] \\
0\ar[r] & \bar i^*\bar a_{d,\sigma}^*{\bar{\mathring q}}_{d,\sigma}^*\mathcal{E}_{d,l}\ar@{..>}[r]\ar@{=}[d] & \mathcal{E}\ar@{..>}[r]\ar[d] & (\bar i^*\mathcal{X}_{d,\sigma})^{\e\e}\ar[r]\ar[d] & 0 \\
0\ar[r] & \bar i^*\bar a_{d,\sigma}^*{\bar{\mathring q}}_{d,\sigma}^*\mathcal{E}_{d,l}\ar[r] & \bar i^*\bar a_{d,\sigma}^*{\bar{\mathring q}}_{d,\sigma}^*\bar\theta_{d,l}^*\mathcal{E}_l\ar[r]\ar[d] & \bar i^*\bar a_{d,\sigma}^*{\bar{\mathring q}}_{d,\sigma}^*\mathcal{T}_{d,l}\ar[r]\ar[d] & 0 \\
& & \mathcal{T}^\e\ar@{=}[r]\ar[d] & \mathcal{T}^\e\ar[d] \\
& & 0 & 0
}
\]
where the last column is the sequence (\ref{sequence2}), the middle row is the pullback of the universal exact sequence of $Q_{d,l}$ via ${\bar{\mathring q}}_{d,\sigma}\bar a_{d,\sigma}\bar i: \mathbb{P}^1\times \mathring{Y}_{d,r,\sigma}\to \mathbb{P}^1\times Q_{d,l}$, and $\mathcal{E}$ is defined to be the kernel of the composition $\bar i^*\bar a_{d,\sigma}^*{\bar{\mathring q}}_{d,\sigma}^*\bar\theta_{d,l}^*\mathcal{E}_l\surj \bar i^*\bar a_{d,\sigma}^*{\bar{\mathring q}}_{d,\sigma}^*\mathcal{T}_{d,l}\surj \mathcal{T}^\e$. The dotted arrows in the first row are the induced maps. Since $\mathcal{T}^\e$ is torsion and flat over $\mathring{Y}_{d,r,\sigma}$ with relative degree $r-l$, the quotient $\bar i^*\bar a_{d,\sigma}^*{\bar{\mathring q}}_{d,\sigma}^*\bar\theta_{d,l}^*\mathcal{E}_l\surj \mathcal{T}^\e$ from the middle column induces a $Q_l$-morphism $\lambda: \mathring{Y}_{d,r,\sigma}\to Q_{r,l}$ such that the pullback of the universal exact sequence of $Q_{r,l}$ via $\bar\lambda$ is the same as the middle column:
\begin{equation}\label{sequence3}
\CD{
0\ar[r] & \bar\lambda^* \mathcal{E}_{r,l}\ar[r]\ar@{=}[d] & \bar\lambda^*\bar\theta_{r,l}^*\mathcal{E}_l\ar[r]\ar@{=}[d] & \bar\lambda^*\mathcal{T}_{r,l}\ar[r]\ar@{=}[d] & 0\\
0\ar[r] & \mathcal{E}\ar[r] & \bar i^*\bar a_{d,\sigma}^*{\bar{\mathring q}}_{d,\sigma}^*\bar\theta_{d,l}^*\mathcal{E}_l\ar[r] & \mathcal{T}^\e\ar[r] & 0
}
\end{equation}
We have a morphism $\phi_{r,l}: Q_{r,l}\to Q_r$ and an identification $\mathcal{E}_{r,l}=\bar\phi_{r,l}^*\mathcal{E}_r$ as in diagram (\ref{DiagramQdr}). Thus we have an identification
\[
\mathcal{E}=\bar\lambda^*\mathcal{E}_{r,l}=\bar\lambda^*\bar\phi_{r,l}^*\mathcal{E}_r
\]
and we can rewrite the first row as
\[
0\to \bar i^*\bar a_{d,\sigma}^*{\bar{\mathring q}}_{d,\sigma}^*\mathcal{E}_{d,l}\to \bar\lambda^*\bar\phi_{r,l}^*\mathcal{E}_r\to (\bar i^*\mathcal{X}_{d,\sigma})^{\e\e}\to 0
\]
Since $(\bar i^*\mathcal{X}_{d,\sigma})^{\e\e}$ is torsion and flat over $\mathring{Y}_{d,r,\sigma}$ with relative degree $d-r$, the above sequence induces a morphism $\mu: \mathring{Y}_{d,r,\sigma}\to Q_{d,r}$ such that the pullback of the universal exact sequence of $Q_{d,r}$ via $\bar\mu$ is the same as the above sequence:
\[
\CD{
0\ar[r] & \bar\mu^*\mathcal{E}_{d,r}\ar[r]\ar@{=}[d] & \bar\mu^*\bar\theta_{d,r}^*\mathcal{E}_r\ar[r]\ar@{=}[d] & \bar\mu^*\mathcal{T}_{d,r}\ar[r]\ar@{=}[d] & 0 \\
0\ar[r] & \bar i^*\bar a_{d,\sigma}^*{\bar{\mathring q}}_{d,\sigma}^*\mathcal{E}_{d,l}\ar[r] & \bar\lambda^*\bar\phi_{r,l}^*\mathcal{E}_r\ar[r] & (\bar i^*\mathcal{X}_{d,\sigma})^{\e\e}\ar[r] & 0
}
\]

Using the identification $\mathcal{T}^\e=\bar\lambda^*\mathcal{T}_{r,l}$ from sequence (\ref{sequence3}), we can rewrite the sequence (\ref{sequence1}) as
\[
0\to  \bar i^*((\bar a_{d,\sigma}^*{\bar{\mathring p}}_{d,\sigma}^*\mathring{\mathcal{X}}_\sigma)(0,1)) \to (\bar i^*\mathcal{X}_{d,\sigma})^{\vee\vee} \to \bar\lambda^*\mathcal{T}_{r,l} \to 0
\]
By Theorem \ref{thm-universal}, the above exact sequence determines a morphism $\nu: \mathring{Y}_{d,r,\sigma}\to P_{r,\sigma}$ such that the pullback of the universal extension of $P_{r,\sigma}$ is the same as the above sequence:
\[
\CD{
0\ar[r] & \bar\nu^*((\bar a_{r,\sigma}^*{\bar{\mathring p}}_{r,\sigma}^*\mathring{\mathcal{X}}_\sigma)(0,1))\ar[r]\ar@{=}[d] & \bar\nu^*\mathcal{X}_{r,\sigma}\ar[r]\ar@{=}[d] & \bar\nu^*\bar a_{r,\sigma}^*{\bar{\mathring q}}_{r,\sigma}^*\mathcal{T}_{r,l}\ar[r]\ar@{=}[d] & 0 \\
0\ar[r] &  \bar i^*((\bar a_{d,\sigma}^*{\bar{\mathring p}}_{d,\sigma}^*\mathring{\mathcal{X}}_\sigma)(0,1))\ar[r] & (\bar i^*\mathcal{X}_{d,\sigma})^{\vee\vee}\ar[r] & \bar\lambda^*\mathcal{T}_{r,l}\ar[r] & 0
}
\]
But because the middle term $(\bar i^*\mathcal{X}_{d,\sigma})^{\vee\vee}$ is locally free, the map $\nu$ actually maps $\mathring{Y}_{d,r,\sigma}$ into $\mathring{P}_{r,\sigma}$. So we obtain a morphism
\[
\mu\times\nu: \mathring{Y}_{d,r,\sigma}\to Q_{d,r}\times_{Q_r}\mathring{P}_{r,\sigma}:=\mathring{R}_{d,r,\sigma}
\]
It is now routine to check that $\mu\times\nu$ is the inverse of $\mathring{\varphi}_{d,r,\sigma}: \mathring{R}_{d,r,\sigma}\to \mathring{Y}_{d,r,\sigma}$, and this complete the proof.
\end{proof}

The above proposition shows that $\mathring{\phi}_{d,r,\sigma}$ is an embedding. So we can identify $\mathring{R}_{d,r,\sigma}$ with the subscheme $\mathring{Y}_{d,r,\sigma}$ of $P_{d,\sigma}$ and identify $\mathring{\phi}_{d,r,\sigma}$ with the inclusion map. We also have $\bar{\mathring{\phi}}_{d,r,\sigma}^*\mathcal{X}_{d,\sigma}=\mathcal{X}_{\mathring{R}_{d,r,\sigma}}$. So the exact sequence (\ref{SES-R(d,r,l)}) can be rewritten as:
\[
0\to {\bar{\mathring q}}_{d,r,\sigma}^*\mathcal{T}_{d,r}\to \bar{\mathring{\phi}}_{d,r,\sigma}^*\mathcal{X}_{d,\sigma}\to {\bar{\mathring p}}_{d,r,\sigma}^*\mathring{\mathcal{X}}_{r,\sigma}\to 0
\]

Since $R_{d,r,\sigma}$ and $P_{d,\sigma}$ are both nonsingular, we can talk about their tangent bundles as well as the normal bundle of the embedding $\mathring{\phi}_{d,r,\sigma}:\mathring{R}_{d,r,\sigma}\embed P_{d,\sigma}$.
\begin{prop}\label{Prop-NormalBundle}
The normal bundle $\mathcal{N}_{\mathring{R}_{d,r,\sigma}/P_{d,\sigma}}$ of the embedding $\mathring{\phi}_{d,r,\sigma}: \mathring{R}_{d,r,\sigma}\embed P_{d,\sigma}$ is isomorphic to $\pi_*\EExt^1({\bar{\mathring q}}_{d,r,\sigma}^*\mathcal{T}_{d,r},{\bar{\mathring p}}_{d,r,\sigma}^*\mathring{\mathcal{X}}_{r,\sigma})$.
\end{prop}
\begin{proof}
Since $P_{d,\sigma}$ is a projective bundle over $\mathring{R}_{d,\sigma}$, the relative cotangent
bundle $\Omega_{P_{d,\sigma}/\mathring{R}_{d,\sigma}}$ fits into the following exact sequence
\[
0\to \Omega_{P_{d,\sigma}/\mathring{R}_{d,\sigma}}\to
(a_{d,\sigma}^*\pi_*\EExt^1({\bar{\mathring{q}}}_{d,\sigma}^*\mathcal{T}_{d,l},
{\bar{\mathring{p}}}_{d,\sigma}^*\mathring{\mathcal{X}}_\sigma))^\vee(-1) \to
\mathcal{O}_{P_{d,\sigma}}\to 0
\]
Pulling the sequence back to $R_{d,r,\sigma}$ via $\phi_{d,r,\sigma}$, we obtain an exact sequence
\begin{equation}\label{SES-Tangent-P}
0\to \phi_{d,r,\sigma}^*\Omega_{P_{d,\sigma}/\mathring{R}_{d,\sigma}}\to
(\phi_{d,r,\sigma}^*a_{d,\sigma}^*\pi_*\EExt^1({\bar{\mathring{q}}}_{d,\sigma}^*\mathcal{T}_{d,l},
{\bar{\mathring{p}}}_{d,\sigma}^*\mathring{\mathcal{X}}_\sigma))^\vee(-1) \to
\phi_{d,r,\sigma}^*\mathcal{O}_{P_{d,\sigma}}\to 0
\end{equation}
We can rewrite
$\phi_{d,r,\sigma}^*a_{d,\sigma}^*\pi_*\EExt^1({\bar{\mathring{q}}}_{d,\sigma}^*\mathcal{T}_{d,l},
{\bar{\mathring{p}}}_{d,\sigma}^*\mathring{\mathcal{X}}_\sigma)$ as
\[
\begin{split}
&\ \quad\phi_{d,r,\sigma}^*a_{d,\sigma}^*\pi_*\EExt^1({\bar{\mathring{q}}}_{d,\sigma}^*\mathcal{T}_{d,l}, {\bar{\mathring{p}}}_{d,\sigma}^*\mathring{\mathcal{X}}_\sigma) \\
&=\pi_*\EExt^1(\bar\phi_{d,r,\sigma}^*\bar a_{d,\sigma}^*{\bar{\mathring{q}}}_{d,\sigma}^*\mathcal{T}_{d,l}, \bar\phi_{d,r,\sigma}^*\bar a_{d,\sigma}^*{\bar{\mathring{p}}}_{d,\sigma}^*\mathring{\mathcal{X}}_\sigma)\\
&=\pi_*\EExt^1(\bar f^*\bar\psi_{d,\ul{r},l}^*\mathcal{T}_{d,l},\bar g^*\mathring{\mathcal{X}}_\sigma)
\end{split}
\]
where
\[
f:=(1\times \mathring{q}_{r,\sigma})(1\times a_{r,\sigma}),\quad g:=\mathring{p}_{d,\sigma}a_{d,\sigma}\phi_{d,r,\sigma}.
\]
Set $QQP:=Q_{d,r}\times_{Q_r}Q_{r,l}\times_{Q_l}\mathring{P}_\sigma$. Since $R_{d,r,\sigma}$ is a
projective bundle over $QQP$, the relative cotangent bundle $\Omega_{R_{d,r,\sigma}/QQP}$ fits into
the following exact sequence
\begin{equation}\label{SES-Tangent-R}
0\to \Omega_{R_{d,r,\sigma}/QQP}\to ((1\times a_{r,\sigma})^*(\psi_{\ul{d},r,l}\times
1)^*\pi_*\EExt^1({\bar{\mathring{q}}}_{r,\sigma}^*\mathcal{T}_{r,l},
{\bar{\mathring{p}}}_{r,\sigma}^*\mathring{\mathcal{X}}_\sigma))^\vee(-1) \to
\mathcal{O}_{R_{d,r,\sigma}}\to 0
\end{equation}
We have
\[
(1\times a_{r,\sigma})^*(\psi_{\ul{d},r,l}\times 1)^*\pi_*\EExt^1({\bar{\mathring{q}}}_{r,\sigma}^*\mathcal{T}_{r,l}, {\bar{\mathring{p}}}_{r,\sigma}^*\mathring{\mathcal{X}}_\sigma)=\pi_*\EExt^1(\bar f^*\bar\psi_{\ul{d},r,l}^*\mathcal{T}_{r,l},\bar g^* \mathring{\mathcal{X}}_\sigma)
\]
based on the equalities $\mathring{q}_{r,\sigma}(\psi_{\ul{d},r,l}\times 1)(1\times a_{r,\sigma})=\psi_{\ul{d},r,l}f$ and $\mathring{p}_{r,\sigma}(\psi_{\ul{d},r,l}\times 1)(1\times a_{r,\sigma})=g$.

We make a diagram
\begin{equation}\label{Diagram-TangentQuotient}
\CD{
 & 0\ar[d] & 0\ar[d] \\
 & \mathcal{K}:\ar@{..>}[d]\ar@{=}[r] & (\pi_*\EExt^1(\bar f^*\bar\psi_{d,r,\ul{l}}^*\mathcal{T}_{d,r},\bar
g^*\mathring{\mathcal{X}}_\sigma))^\vee(-1)\ar[d] \\
0\ar[r] & \phi_{d,r,\sigma}^*\Omega_{P_{d,\sigma}/\mathring{R}_{d,\sigma}}\ar[r]\ar@{..>}[d] &
(\pi_*\EExt^1(\bar f^*\bar\psi_{d,\ul{r},l}^*\mathcal{T}_{d,l},\bar
g^*\mathring{\mathcal{X}}_\sigma))^\vee(-1)\ar[r]\ar[d]\ar@{}[dr]|-{\textcircled{\,1}} &
\phi_{d,r,\sigma}^*\mathcal{O}_{P_{d,\sigma}}\ar[r]\ar@{=}[d] & 0 \\
0\ar[r] & \Omega_{R_{d,r,\sigma}/QQP}\ar[r]\ar[d] & (\pi_*\EExt^1(\bar
f^*\bar\psi_{\ul{d},r,l}^*\mathcal{T}_{r,l},\bar g^*
\mathring{\mathcal{X}}_\sigma))^\vee(-1)\ar[r]\ar[d] &
\mathcal{O}_{R_{d,r,\sigma}}\ar[r] & 0 \\
& 0 & 0
}%
\end{equation}
where
\begin{itemize}
\item the last row is the exact sequence (\ref{SES-Tangent-P}),
\item the middle row is the exact sequence (\ref{SES-Tangent-R}), and
\item the middle column is obtained by applying $\pi_*\EExt^1(\bar f^*(-),\bar
g^*\mathring{\mathcal{X}}_\sigma)$ to the exact sequence (\ref{SES-Torsions}), then taking dual,
and lastly twisting by $\mathcal{O}_{R_{d,r,\sigma}}(-1)$.
\end{itemize}
{ The commutativity of the rectangle \textcircled{1} in the diagram (\ref{Diagram-TangentQuotient})
follows from diagram (\ref{Diagram-R(d,r,l)}) and the identifications (\ref{Identification-phi}).}
Thus, we have induced maps (the dotted arrows) in the first column and the first column is exact.
Restricting the first column to $\mathring{R}_{d,r,\sigma}$, we obtain an exact sequence
\[
0\to \mathring{\mathcal{K}}\to
\mathring{\phi}_{d,r,\sigma}^*\Omega_{P_{d,\sigma}/\mathring{R}_{d,\sigma}}\to
\Omega_{\mathring{R}_{d,r,\sigma}/QQP}\to 0
\]
where $\mathring{\mathcal{K}}:=\mathcal{K}|_{\mathring{R}_{d,r,\sigma}}$. We have
\[
\mathring{\mathcal{K}}=(\pi_*\EExt^1(\bar{\mathring{f}}^*\bar\psi_{d,r,\ul{l}}^*\mathcal{T}_{d,r},
\bar{\mathring{g}}^*\mathring{\mathcal{X}}_\sigma))^\vee(-1)
\]
where $\mathring{f}$ and $\mathring{g}$ are the restrictions of $f$ and $g$ to $\mathring{R}_{d,r,\sigma}$.

Recall that $\psi_{d,r,l}: Q_{d,r,l}\to Q_{d,l}$ is a finite morphism. Since $\Bbbk$ is assumed to be of characteristic 0, the function field extension $K(Q_{d,r,l})/K(Q_{d,l})$ is separable. Therefore, we have an exact sequence of relative cotangent sheaves
\[
0\to \psi_{d,\ul{r},l}^*\Omega_{Q_{d,l}}\to \Omega_{Q_{d,r,l}}\to \Omega_{Q_{d,r,l}/Q_{d,l}}\to 0
\]
where both $\Omega_{Q_{d,l}}$ and $\Omega_{Q_{d,r,l}}$ are locally free of the same rank while $\Omega_{Q_{d,r,l}/Q_{d,l}}$ is torsion.

We form the following diagram
\[
\CD{%
& & & 0\ar[d]\\
\psi_{\ul{d},r,l}^*\theta_{r,l}^*\Omega_{Q_l}\ar@{=}[rd] & 0\ar[d] & 0\ar[r]\ar[d] &
(\pi_*\HHom(\bar\psi_{d,r,\ul{l}}^*\mathcal{E}_{d,r},\bar\psi_{\ul{d},r,l}^*\mathcal{T}_{r,l}))^\vee
\ar[d]
\\
0\ar[r] & \psi_{d,\ul{r},l}^*\theta_{d,l}^*\Omega_{Q_l}\ar[r]\ar[d] &
\psi_{d,\ul{r},l}^*\Omega_{Q_{d,l}}\ar[r]\ar[d] &
\psi_{d,\ul{r},l}^*\Omega_{Q_{d,l}/Q_l}\ar[r]\ar[d] & 0 \\
0\ar[r] & \psi_{\ul{d},r,l}^*\Omega_{Q_{r,l}}\ar[r]\ar[d] & \Omega_{Q_{d,r,l}}\ar[r]\ar[d] &
\Omega_{Q_{d,r,l}/Q_{r,l}}\ar[r]\ar[d] & 0 \\
& \psi_{\ul{d},r,l}^*\Omega_{Q_{r,l}/Q_l}\ar@{..>}[r]\ar[d] & \Omega_{Q_{d,r,l}/Q_{d,l}}\ar[r]\ar[d] & 0 \\
& 0 & 0
}%
\]
Applying the Snake Lemma to the third row and the fourth row of the above diagram, we obtain an
exact sequence by connecting the second row with the last row:
\[
0\to
(\pi_*\HHom(\bar\psi_{d,r,\ul{l}}^*\mathcal{E}_{d,r},\bar\psi_{\ul{d},r,l}^*\mathcal{T}_{r,l}))^\vee
\to \psi_{\ul{d},r,l}^*\Omega_{Q_{r,l}/Q_l} \to \Omega_{Q_{d,r,l}/Q_{d,l}} \to 0
\]
This sequence fits into the following commutative diagram
\[
\CD{
0\ar[d] & 0 \ar[d] \\ 
(\pi_*\HHom(\bar\psi_{d,r,\ul{l}}^*\mathcal{E}_{d,r},\bar\psi_{\ul{d},r,l}^*\mathcal{T}_{r,l}))^\vee
\ar@{=}[r]\ar[d] &
(\pi_*\HHom(\bar\psi_{d,r,\ul{l}}^*\mathcal{E}_{d,r},\bar\psi_{\ul{d},r,l}^*\mathcal{T}_{r,l}))^\vee
\ar[d]
\\
\psi_{\ul{d},r,l}^*\Omega_{Q_{r,l}/Q_l} \ar@{=}[r]\ar[d] &
(\pi_*\HHom(\bar\psi_{\ul{d},r,l}^*\mathcal{E}_{r,l},\bar\psi_{\ul{d},r,l}^*\mathcal{T}_{r,l}))^\vee
\ar[d]
\\
\Omega_{Q_{d,r,l}/Q_{d,l}}\ar@{:}[r]\ar[d] &
(\pi_*\EExt^1(\bar\psi_{d,r,\ul{l}}^*\mathcal{T}_{d,r},\bar\psi_{\ul{d},r,l}^*\mathcal{T}_{r,l}))^\e\ar[d]
\\
0 & 0
}%
\]
where the identification in the second row follows from the canonical identification
$\mathcal{T}_{Q_{r,l}/Q_l}=\pi_*\HHom(\bar\psi_{d,r,\ul{l}}^*\mathcal{E}_{d,r},\bar\psi_{\ul{d},r,l}^*\mathcal{T}_{r,l})$.
So we obtain an identification of the quotients
\[
\Omega_{Q_{d,r,l}/Q_{d,l}}=(\pi_*\EExt^1(\bar\psi_{d,r,\ul{l}}^*\mathcal{T}_{d,r},\bar\psi_{\ul{d},r,l}^*\mathcal{T}_{r,l}))^\e
\]

Since $Q_{d,r,l}$ is smooth over $Q_l$, $QQP=Q_{d,r,l}\times_{Q_l}\mathring{P}_\sigma$ is smooth over $\mathring{P}_\sigma$. The morphism $\psi_{d,\ul{r},l}\times 1: QQP\to \mathring{R}_{d,\sigma}$ is obtained from $\psi_{d,\ul{r},l}$ by the base change $\mathring{q}_{d,\sigma}:\mathring{R}_{d,\sigma}\to Q_{d,l}$, hence we have the following commutative diagram
\[
\CD{
0\ar[r] & (\psi_{d,\ul{r},l}\times 1)^*\Omega_{\mathring{P}_{d,\sigma}}\ar[r]\ar@{=}[d] & \Omega_{QQP} \ar[r]\ar@{=}[d] & \Omega_{QQP/\mathring{P}_{d,\sigma}} \ar[r]\ar@{=}[d] & 0 \\
0\ar[r] & (1\times\mathring{q}_{d,\sigma})^*\psi_{d,\ul{r},l}^*\Omega_{Q_{d,l}}\ar[r] & (1\times\mathring{q}_{d,\sigma})^*\Omega_{Q_{d,r,l}}\ar[r] & (1\times\mathring{q}_{d,\sigma})^*\Omega_{Q_{d,r,l}/Q_{d,l}} \ar[r] & 0
}
\]

The conormal bundle $\mathcal{N}^\vee_{\mathring{R}_{d,r,\sigma}/P_{d,\sigma}}$ fits into the
following commutative diagram
\[
\CD{%
& & 0\ar[d] & 0\ar[d]\\
& 0\ar[r]\ar[d] & \mathcal{N}_{\mathring{R}_{d,r,\sigma}/P_{d,\sigma}}^\vee\ar[r]\ar[d] &  \mathring{\mathcal{K}}\ar[d] \\
0\ar[r] & \mathring{\phi}_{d,r,\sigma}^*a_{d,\sigma}^*\Omega_{\mathring{R}_{d,\sigma}}\ar[r]\ar[d]
& \mathring{\phi}_{d,r,\sigma}^*\Omega_{P_{d,\sigma}}\ar[r]\ar[d] &
\mathring{\phi}_{d,r,\sigma}^*\Omega_{P_{d,\sigma}/\mathring{R}_{d,\sigma}}\ar[r]\ar[d] & 0\\
0\ar[r] & (1\times \mathring{a}_{r,\sigma})^*\Omega_{QQP}\ar[r]\ar[d] &
\Omega_{\mathring{R}_{d,r,\sigma}} \ar[r]\ar[d]
& \Omega_{\mathring{R}_{d,r,\sigma}/QQP}\ar[r]\ar[d] & 0\\
 & \mathring{f}^*\Omega_{Q_{d,r,l}/Q_{d,l}}\ar[r]\ar[d] & 0 & 0 \\
& 0
}%
\]
Applying the Snake Lemma to the third and fourth rows, we obtain an exact sequence by connecting
the second row with the fifth row:
\[
0 \to \mathcal{N}_{\mathring{R}_{d,r,\sigma}/P_{d,\sigma}}^\vee \to \mathring{\mathcal{K}} \to
\mathring{f}^*\Omega_{Q_{d,r,l}/Q_{d,l}} \to 0
\]

The above sequence fits into the following commutative diagram
\[
\CD{
0\ar[d] & 0\ar[d] \\ 
\mathcal{N}_{\mathring{R}_{d,r,\sigma}/P_{d,\sigma}}^\vee\ar@{:}[r]\ar[d] & (\pi_*\EExt^1(\bar{\mathring{f}}^*\bar\psi_{d,r,\ul{l}}^*\mathcal{T}_{d,r}, \bar{\mathring{p}}_{d,r,\sigma}^*\mathring{\mathcal{X}}_{r,\sigma}))^\vee  \ar[d]\\
\mathring{\mathcal{K}} \ar@{=}[r]\ar[d] & (\pi_*\EExt^1(\bar{\mathring{f}}^*\bar\psi_{d,r,\ul{l}}^*\mathcal{T}_{d,r}, \bar{\mathring{g}}^*\mathring{\mathcal{X}}_\sigma))^\vee(-1) \ar[d] \\
\mathring{f}^*\Omega_{Q_{d,r,l}/Q_{d,l}}\ar@{=}[r]\ar[d] & (\pi_*\EExt^1(\bar{\mathring{f}}^*\bar\psi_{d,r,\ul{l}}^*\mathcal{T}_{d,r}, \bar{\mathring{f}}^*\bar\psi_{\ul{d},r,l}^*\mathcal{T}_{r,l}))^\e\ar[d] \\
0 & 0
}
\]
where the second column is obtained as follows: first we can rewrite the last row of {diagram (\ref{Diagram-R(d,r,l)})} as
\[
0\to (\bar g^*\mathring{\mathcal{X}}_\sigma)(0,1)\to \bar p_{d,r,\sigma}^*\mathcal{X}_{r,\sigma}\to \bar f^*\bar \psi_{\ul{d},r,l}\mathcal{T}_{r,l}\to 0;
\]
next the restriction to $\mathbb{P}^1\times\mathring{R}_{d,r,\sigma}$ is
\[
0\to (\bar{\mathring{g}}^*\mathring{\mathcal{X}}_\sigma)(0,1)\to \bar{\mathring{p}}_{d,r,\sigma}^*\mathring{\mathcal{X}}_{r,\sigma}\to \bar{\mathring{f}}^*\bar \psi_{\ul{d},r,l}\mathcal{T}_{r,l}\to 0;
\]
and next applying $\pi_*\HHom(\bar{\mathring{f}}^*\bar\psi_{d,r,\ul{l}}^*\mathcal{T}_{d,r},-)$ to
the above sequence, and lastly taking dual, we obtain the second column in the diagram. Thus there
is an induced identification as in the first row, which gives the natural identification
\[
\mathcal{N}_{\mathring{R}_{d,r,\sigma}/P_{d,\sigma}}=\pi_*\EExt^1(\bar{\mathring{f}}^*\bar\psi_{d,r,\ul{l}}^*\mathcal{T}_{d,r}, \bar{\mathring{p}}_{d,r,\sigma}^*\mathring{\mathcal{X}}_{r,\sigma})=\pi_*\EExt^1(\bar{\mathring{q}}^*_{d,r,\sigma}\mathcal{T}_{d,r}, \bar{\mathring{p}}_{d,r,\sigma}^*\mathring{\mathcal{X}}_{r,\sigma})
\]
\end{proof}

\section{The modular interpretation}

The compactification $\widetilde Q_d$ is obtained by successively blowing up the Quot scheme $Q_d$ along $Z_{d,0}, \cdots, Z_{d,d-1}$.
We illustrate the process in the following diagram
\[\label{blowupdiagram}
\begin{array}{cccccccccccl}
Z_{d,0}^{d-1} &  & Z_{d,1}^{d-1} & & Z_{d,2}^{d-1} & & \cdots &         & Z_{d,d-1}^{d-1} & \subset & Q_d^{d-1}=\widetilde Q_d & \\
\downarrow & & \downarrow & & \downarrow &        &  &         &  \downarrow &         & \downarrow &\leftarrow\text{ along }{Z_{d,d-1}^{d-2}} \\
\vdots     & & \vdots  & & \vdots   &        & \vdots&         &  \vdots     &         & \vdots & \\
\downarrow & & \downarrow & & \downarrow &        &  &         &  \downarrow &         & \downarrow &\leftarrow\text{ along }{Z_{d,2}^1} \\
Z_{d,0}^1 &  & Z_{d,1}^1 & & Z_{d,2}^1 & \subset & \cdots & \subset & Z_{d,d-1}^1 & \subset & Q_d^1 & \\
\downarrow & & \downarrow & & \downarrow &        &  &         &  \downarrow &         & \downarrow &\leftarrow\text{ along }{Z_{d,1}^0} \\
Z_{d,0}^0 &  & Z_{d,1}^0 & \subset & Z_{d,2}^0 & \subset & \cdots & \subset & Z_{d,d-1}^0 & \subset & Q_d^0 & \\
\downarrow & & \downarrow & & \downarrow &        &  &         &  \downarrow &         & \downarrow &\leftarrow\text{ along }{Z_{d,0}} \\
Z_{d,0} & \subset & Z_{d,1} & \subset & Z_{d,2} & \subset & \cdots & \subset & Z_{d,d-1} & \subset
& Q_d &
\end{array}
\]
Here, inductively,  $Z_{d,0}^j, \cdots, Z_{d,j}^j$ are the exceptional divisors
created by the sequence of blowups  $Q_d^j \to Q_d$ ($0 \le j \le d-1$); the nested subschemes
$Z_{d,j+1}^j \subset \cdots \subset Z_{d,d-1}^j$
are the proper transforms of
the subschemes $Z_{d,j+1}^{j-1} \subset \cdots \subset Z_{d,d-1}^{j-1}$ (respectively); these are the subschemes lining up
to be  blown-up in the next steps. Below, we will provide modular meanings to the points of $Z_{d,0}^j, \cdots, Z_{d,j}^j$ for all $j$. Thus
every intermediate space $Q_d^j$, $j=0,\dots,d-1$, also compactification of $\mathring{Q}_d$, admits parameter space interpretation
(Proposition. \ref{stratum}).
The case $Q_d^{d-1}$ is our final space $\widetilde Q_d$ (Corollary \ref{final-interp}).


To proceed, we introduce the following
\[
\hat Z_{d,r}^l:=\left\{\begin{array}{ll} \varnothing, & \text{ if } r=l \\ Z_{d,r}^l, & \text{ if } l<r<d \end{array}\right.,\quad \mathring{Z}_{d,r}^l:=Z_{d,r}^l\setminus \hat Z_{d,r-1}^l, \quad r=l+1,\dots,d-1
\]

Let $d>r\geq 0$. For each subsequence $\sigma\subset[r]:=(r,r-1,\cdots,1,0)\in\Sigma$, we set
\[
E_{d,\sigma}^r:=\bigcap_{i\in \sigma}Z_{d,i}^r\setminus \bigcup_{i\in [r]\setminus \sigma}Z_{d,i}^r.
\]
Then there is a stratification of $Q_d^r$:
\[
Q_d^r=(Q_d^r\setminus \bigcup_{i\in[r]} Z_{d,i}^r)\sqcup \bigsqcup_{\sigma\subset[r]}E_{d,\sigma}^r =(Q_d\setminus Z_{d,r})\sqcup \bigsqcup_{\sigma\subset[r]}E_{d,\sigma}^r.
\]

\begin{lemma}\label{lem_E}
Let $\tau\in \Sigma$, $t=\lt(\tau)$. We have
\begin{enumerate}
\item $E_{d,\tau}^j=E_{d,\tau}^m\setminus Z_{d,j}^m$, for any $m$ and $j$ with $t\leq m<j<d$.
\item if $l>t$ and $\sigma=(l,\tau)$, then $E_{d,\sigma}^l$ is the exceptional divisor of the blowup of $E_{d,\tau}^t\setminus \hat Z_{d,l-1}^t$ along $E_{d,\tau}^t\cap \mathring{Z}_{d,l}^t$.
\end{enumerate}
\end{lemma}
\begin{proof}
{
(1) Consider the composite blowup $b: Q_d^j\to Q_d^m$, $m<j$. We have $b^{-1}(Z_{d,i}^m)=Z_{d,i}^j$ for $i=0,\dots,m$, and $b^{-1}(Z_{d,j}^m)=\bigcup_{i=m+1}^j Z_{d,i}^j$. Moreover, $b$ is an isomorphism away from $Z_{d,j}^m$. These facts give us identifications
\[
Q_d^j\setminus b^{-1}(Z_{d,j}^m)=Q_d^m\setminus Z_{d,j}^m,\text{ and } Z_{d,i}^j\setminus b^{-1}(Z_{d,j}^m) = Z_{d,i}^m\setminus Z_{d,j}^m, \text{ for }i=0,\dots,m
\]
Therefore
\[
\begin{split}
E_{d,\tau}^j& =\bigcap_{i\in\tau}Z_{d,i}^j\setminus \bigcup_{i\in[j]\setminus\tau} Z_{d,i}^j=\bigg(\bigcap_{i\in\tau}Z_{d,i}^j\setminus \bigcup_{i\in[m]\setminus\tau} Z_{d,i}^j\bigg)\setminus \bigcup_{i=m+1}^j
Z_{d,i}^j \\
& = \bigg(\bigcap_{i\in\tau}Z_{d,i}^j\setminus \bigcup_{i\in[m]\setminus\tau} Z_{d,i}^j\bigg)\setminus b^{-1}(Z_{d,j}^m) \\
& = \bigg(\bigcap_{i\in\tau}Z_{d,i}^m\setminus \bigcup_{i\in[m]\setminus\tau} Z_{d,i}^m\bigg)\setminus Z_{d,j}^m = E_{d,\tau}^m\setminus Z_{d,j}^m
\end{split}
\]

(2) By definition,
\[
E_{d,\sigma}^l=\bigcap_{i\in\sigma}Z_{d,i}^l\setminus \bigcup_{i\in [l]\setminus\sigma}Z_{d,i}^l = Z_{d,l}^l\cap \bigcap_{i\in\tau}Z_{d,i}^l\setminus \bigcup_{i\in [l]\setminus\sigma}Z_{d,i}^l
\]
We know $Z_{d,l}^l$ is the exceptional divisor of the blowup $Q_d^l\to Q_d^{l-1}$, which is along $Z_{d,l}^{l-1}$. Since $(\bigcap_{i\in\tau}Z_{d,i}^l\setminus \bigcup_{i\in [l]\setminus\sigma}Z_{d,i}^l)\subset Q_d^l$ is exactly the preimage of  $\bigcap_{i\in\tau}Z_{d,i}^{l-1}\setminus \bigcup_{i\in [l]\setminus\sigma}Z_{d,i}^{l-1}\subset Q_d^{l-1}$ under the blowup, we have that $E_{d,\sigma}^l$ is the exceptional divisor of the (induced) blowup of $\bigcap_{i\in\tau}Z_{d,i}^{l-1}\setminus \bigcup_{i\in [l]\setminus\sigma}Z_{d,i}^{l-1}$ along $Z_{d,l}^{l-1}\cap\bigcap_{i\in\tau}Z_{d,i}^{l-1}\setminus \bigcup_{i\in [l]\setminus\sigma}Z_{d,i}^{l-1}$. Note that
\[
\bigcap_{i\in\tau}Z_{d,i}^{l-1}\setminus \bigcup_{i\in [l]\setminus\sigma}Z_{d,i}^{l-1}=\bigcap_{i\in\tau}Z_{d,i}^{l-1}\setminus \bigg(\bigcup_{i\in [t]\setminus\tau}Z_{d,i}^{l-1}\cup \bigcup_{i=t+1}^{l-1}Z_{d,i}^{l-1}\bigg)=E_{d,\tau}^{l-1}\setminus \bigcup_{i=t+1}^{l-1}Z_{d,i}^{l-1}
\]
If $l-1=t$, then $E_{d,\sigma}^l$ is the exceptional divisor of the blowup of $E_{d,\tau}^t$ along $Z_{d,t+1}^t$, and we are done. Now suppose $l-1>t$. Then under the identification $Q_d^{l-1}\setminus \bigcup_{i=t+1}^{l-1}Z_{d,i}^{l-1}=Q_d^t\setminus Z_{d,l-1}^t$, we have identifications $E_{d,\tau}^{l-1}\setminus \bigcup_{i=t+1}^{l-1}Z_{d,i}^{l-1} = E_{d,\tau}^t\setminus Z_{d,l-1}^t$ and
$$Z_{d,l}^{l-1}\setminus \bigcup_{i=t+1}^{l-1}Z_{d,i}^{l-1}= Z_{d,l}^t\setminus Z_{d,l-1}^t=\mathring{Z}_{d,l}^t.$$}
\end{proof}

\begin{prop}
There is a collection of isomorphisms
\[
i_{d,\sigma}: P_{d,\sigma}\stackrel{\sim}{\to} E_{d,\sigma}^l,
\]
one for each $\sigma\in\Sigma$ with $l:=\lt(\sigma)<d$, such that the following properties hold:
\begin{enumerate}
\item $i_{d,\sigma}$ maps $Y_{d,r,\sigma}$ onto $E_{d,\sigma}^l\cap Z_{d,r}^l$ for all $r$, $l<r<d$;;
\item $i_{d,\sigma}$ maps $\mathring{Y}_{d,r,\sigma}$ isomorphically onto $E_{d,\sigma}^l\cap \mathring{Z}_{d,r}^l$ for all $r$, $l<r<d$;
\item The following diagram commutes:
\[
\xymatrix{
P_{d,\sigma}\ar[d]\ar[rr]^{i_{d,\sigma}}_{\sim} & & E_{d,\sigma}^l\ar[d]\\
\mathring{R}_{d,l,\tau}\ar[r]^{\sim} & \mathring{Y}_{d,l,\tau}\ar[r]^-{\sim}_-{i_{d,\tau}} & E_{d,\tau}^t\cap \mathring{Z}_{d,l}^t
}
\]
where $\sigma=(l,\tau)$ and $t=\lt(\tau)$;
\item Let $e_{d,\sigma}: P_{d,\sigma}\embed Q_d^l$ be the embedding obtained through the composition $P_{d,\sigma}\stackrel{\sim}{\to} E_{d,\sigma}^l\embed Q_d^l$. Then the following diagram commutes
\[
\xymatrix@C=3pc{
Q_{d,r}\times_{Q_r}P_{r,\sigma}\ar@{^(->}[r]^-{1\times e_{r,\sigma}}\ar[d]_{\phi_{d,r,\sigma}} & Q_{d,r}\times_{Q_r}Q_{r}^l \ar[d]^{\phi_{d,r}^l} \\
P_{d,\sigma}\ar@{^(->}[r]^-{e_{d,\sigma}} & Q_{d,\sigma}^l
}
\]
for all $r$, $l<r<d$.
\end{enumerate}
\end{prop}
\begin{proof}
We prove by constructing the isomorphisms $i_{d,\sigma}$ explicitly, and this is done by induction on the length of $\sigma$. We first deal with the base case: $\sigma\in \Sigma_1$ or $\sigma=(l)$. In this case, we construct isomorphisms
\[
P_{d,l}\stackrel{\sim}{\to} E_{d,l}^l=Z_{d,l}^l\setminus \bigcup_{0\leq i\leq l-1}Z_{d,i}^l
\]
which map $Y_{d,r,l}$ onto $E_{d,l}^l\cap Z_{d,r}^l$ and map $\mathring{Y}_{d,r,l}$ isomorphically onto $E_{d,l}^l\cap \mathring{Z}_{d,r}^l$ for all $d$, $r$ and $l$ such that $l<r<d$.


Recall that $Z_{d,l}^l$ is the exceptional divisor in the blowup $b: Q_d^l\to Q_d^{l-1}$ along $Z_{d,l}^{l-1}$. We have $b^{-1}(Z_{d,i}^{l-1})=Z_{d,i}^l$ for $i=0,\dots, l-1$, hence
\[
b: Q_d^l\setminus \bigcup_{i=0}^{l-1}Z_{d,i}^l\to Q_d^{l-1}\setminus \bigcup_{i=0}^{l-1}Z_{d,i}^{l-1}
\]
is the blowup along $Z_{d,l}^{l-1}\setminus \bigcup_{i=0}^{l-1}Z_{d,i}^{l-1}$ with exceptional divisor $Z_{d,l}^l\setminus \bigcup_{i=0}^{l-1}Z_{d,i}^l=E_{d,l}^l$. Note that $Q_d^{l-1}\setminus \bigcup_{i=0}^{l-1}Z_{d,i}^{l-1}=Q_d\setminus Z_{d,l-1}$ and $Z_{d,l}^{l-1}\setminus \bigcup_{i=0}^{l-1}Z_{d,i}^{l-1}=Z_{d,l}\setminus Z_{d,l-1}=\mathring{Z}_{d,l}$. Therefore $E_{d,l}^l$ is the projective bundle $\mathbb{P}(\mathcal{N}_{\mathring{Z}_{d,l}/Q_d})$ over $\mathring{Z}_{d,l}$. Since $\mathring{Z}_{d,l}\simeq \mathring{Q}_{d,l}$ and by Proposition \ref{Prop-NormalBundle0}
\[
\mathcal{N}_{\mathring{Q}_{d,l}/Q_d}=\pi_*\EExt^1(\mathcal{T}_{d,l},\bar\theta_{d,l}^*\mathcal{F}_l) |_{\mathring{Q}_{d,l}}=\pi_*\EExt^1(\bar{\mathring{q}}^*_{d,l}\mathcal{T}_{d,l}, \bar{\mathring{p}}_{d,l}^*\mathring{\mathcal{X}}_{l}),
\]
we know $E_{d,l}^l$ is isomorphic to the projective bundle
\[
P_{d,l}=\mathbb{P}(\pi_*\EExt^1(\bar{\mathring{q}}^*_{d,l}\mathcal{T}_{d,l}, \bar{\mathring{p}}_{d,l}^*\mathring{\mathcal{X}}_{l}))
\]
over $\mathring{Q}_{d,l}$. So we obtain an embedding $e_{d,l}: P_{d,l}\embed Q_d^l$ through the composition
\[
P_{d,l}\stackrel{\sim}{\to} E_{d,l}^l\embed Q_d^l.
\]
For each $r$ with $l<r<d$, we have a commtative diagram
\[
\xymatrix@C=4pc{
Q_{d,r}\times_{Q_{r}}P_{r,l}\ar@{^(->}[r]^-{1\times e_{r,l}}\ar[d]_{\phi_{d,r,l}} &  Q_{d,r}\times_{Q_{r}}Q_{r}^l\ar[d]^{\phi_{d,r}^l}\\
P_{d,l}\ar@{^(->}[r]^-{e_{d,l}} & Q_d^l
}
\]
We have that $\Img(\phi_{d,r,l})=Y_{d,r,l}$, that $\Img(1\times e_{r,l})=Q_{d,r}\times_{Q_r}E_{r,l}^l$, and that $\phi_{d,r}^l$ maps $Q_{d,r}\times_{Q_r}E_{r,l}^l$ onto $E_{d,l}^l\cap Z_{d,r}^l$. It follows that $e_{d,l}$ maps $Y_{d,r,l}$ onto $E_{d,l}^l\cap Z_{d,r}^l$.



Next, we show that $e_{d,l}: P_{d,l}\to Q_d^l$ maps $\mathring{Y}_{d,r,l}$ isomorphically onto $E_{d,l}^l\cap \mathring{Z}_{d,r}^l$ for $l<r<d$. We have a commutative diagram
\[
\xymatrix@C=2pc{
Q_{d,r}\times_{Q_{r}}\mathring{P}_{r,l}\ar@{^(->}[r]^-{1\times e_{r,l}}\ar@{^(->}[d]_{\phi_{d,r,l}} & Q_{d,r}\times_{Q_{r}}(Q_{r}^l\setminus Z_{r,r-1}^l)\ar[d]_{\phi_{d,r}^l}\ar@{=}[r] & Q_{d,r}\times_{Q_r}(Q_r^{r-1}\setminus \ds\bigcup_{i=l+1}^{r-1} Z_{r,i}^{r-1})\ar@{^(->}[d]^{\phi_{d,r}^{r-1}} \\
P_{d,l}\setminus Y_{d,r-1,l}\ar@{^(->}[r]^{e_{d,l}} & Q_d^l\setminus Z_{d,r-1}^l \ar@{=}[r] &
Q_d^{r-1}\setminus \ds\bigcup_{i=l+1}^{r-1}Z_{d,i}^{r-1}
}
\]
We see that $e_{r,l}$ maps $\mathring{P}_{r,l}$ isomorphically onto $E_{r,l}^l\setminus Z_{r,r-1}^l$ in $Q_r^l\setminus Z_{r,r-1}^l$, $E_{r,l}^l\setminus Z_{r,r-1}^l$, which can be identified with $E_{r,l}^{r-1}\subset Q_r^{r-1}\setminus\bigcup_{i=l+1}^{r-1} Z_{r,i}^{r-1}$ {(by Lemma \ref{lem_E})}, and $\phi_{d,r}^{r-1}$ maps $Q_{d,r}\times_{Q_r}E_{r,l}^{r-1}$ isomorphically onto $E_{d,l}^{r-1}\cap Z_{d,r}^{r-1}$, which is identified with $E_{d,l}^l\cap \mathring{Z}_{d,r}^l$. On the other hand, $\phi_{d,r,l}$ maps $Q_{d,r}\times_{Q_{r}}\mathring{P}_{r,l}$ isomorphically onto $\mathring{Y}_{d,r,l}\subset P_{d,l}\setminus Y_{d,r-1,l}$. It follows that $e_{d,l}$ maps $\mathring{Y}_{d,r,l}$ isomorphically onto $E_{d,l}^l\cap \mathring{Z}_{d,r}^l$. Thus the case that $\sigma\in\Sigma_1$ is constructed.


Suppose we have constructed the isomorphisms for all $\sigma\in\Sigma_m$ for some $m$. We now construct the isomorphisms for $\sigma\in \Sigma_{m+1}$. Now write $\sigma=(l,\tau)$ and let $t=\lt(\tau)$. By induction hypothesis, for any $d>t$, we have an isomorphism $P_{d,\tau}\stackrel{\sim}{\to} E_{d,\tau}^t$ which maps $Y_{d,r,\tau}$ onto $E_{d,\tau}^t\cap Z_{d,r}^t$ and maps  $\mathring{Y}_{d,r,\tau}$ isomorphically onto $E_{d,\tau}^t\cap \mathring{Z}_{d,r}^t$. Let $d>l$. {By Lemma \ref{lem_E}}, $E_{d,\sigma}^l$ is the exceptional divisor of the blowup of $E_{d,\tau}^t\setminus \tilde Z_{d,l-1}^t$ along $E_{d,\tau}^t\cap \mathring{Z}_{d,l}^t$. Since the isomorphism $P_{d,\tau}\stackrel{\sim}{\to} E_{d,\tau}^t$ maps $\mathring{Y}_{d,l,\tau}$ isomorphically onto $E_{d,\tau}^t\cap \mathring{Z}_{d,l}^t$, we have an isomorphism
\[
\Bl_{\mathring{Y}_{d,l,\tau}}(P_{d,\tau}\setminus Y_{d,l-1,\tau})\stackrel{\sim}{\to}\Bl_{E_{d,\tau}^t\cap \mathring{Z}_{d,l}^t} (E_{d,\tau}^t\setminus \tilde Z_{d,l-1}^t)
\]
which maps the exceptional divisor $E_{d,\sigma}^l$ of the blowup $\Bl_{\mathring{Y}_{d,l,\tau}}(P_{d,\tau}\setminus Y_{d,l-1,\tau})$ isomorphically onto the exceptional divisor $E_{d,\sigma}^l$ of $\Bl_{E_{d,\tau}^t\cap \mathring{Z}_{d,l}^t}(E_{d,\tau}^t\setminus Z_{d,l-1}^t)$. On the other hand, the exceptional divisor of $\Bl_{\mathring{Y}_{d,l,\tau}}(P_{d,\tau}\setminus Y_{d,l-1,\tau})$ is isomorphic to the projective normal bundle of $\mathring{Y}_{d,l,\tau}$ in $P_{d,\tau}\setminus Y_{d,l-1,\tau}$ or just in $P_{d,\tau}$. We know that $\phi_{d,l,\tau}: Q_{d,l}\times_{Q_l}P_{l,\tau}\to P_{d,\tau}$ maps $Q_{d,l}\times_{Q_l}\mathring{P}_{l,\tau}$ isomorphically onto $\mathring{Y}_{d,l,\tau}$, and by Proposition \ref{Prop-NormalBundle} the normal bundle of $Q_{d,l}\times_{Q_l}\mathring{P}_{l,\tau}$ in $P_{d,\tau}$ is $\pi_*\EExt^1({\bar{\mathring q}}_{d,l,\tau}^*\mathcal{T}_{d,l},{\bar{\mathring p}}_{d,l,\tau}^*\mathring{\mathcal{X}}_{l,\tau})$. Hence the projective normal bundle of $\mathring{Y}_{d,l,\tau}$ in $P_{d,\tau}$ is isomorphic to $\mathbb{P}(\pi_*\EExt^1({\bar{\mathring q}}_{d,l,\tau}^*\mathcal{T}_{d,l},{\bar{\mathring p}}_{d,l,\tau}^*\mathring{\mathcal{X}}_{l,\tau}))=P_{d,l,\tau}=P_{d,\sigma}$. Thus we obtain an isomorphism $P_{d,\sigma}\stackrel{\sim}{\to} E_{d,\sigma}^l$. Next we show that this isomorphism maps $Y_{d,r,\sigma}$ onto $E_{d,\sigma}^l\cap Z_{d,r}^l$ and maps $\mathring{Y}_{d,r,\sigma}$ isomorphically onto $E_{d,\sigma}\cap \mathring{Z}_{d,r}^l$.

Let $e_{d,\sigma}: P_{d,\sigma}\embed Q_{d}^l$ denote the embedding obtained from the composition $P_{d,\sigma}\stackrel{\sim}{\to}E_{d,\sigma}^l\embed Q_{d,\sigma}^l$ for each $d>l$. Let $d>r>l$. Then we have a commutative diagram
\[
\xymatrix@C=3pc{
Q_{d,r}\times_{Q_r}P_{r,\sigma}\ar@{^(->}[r]^-{1\times e_{r,\sigma}}\ar[d]_{\phi_{d,r,\sigma}} & Q_{d,r}\times_{Q_r}Q_{r}^l \ar[d]^{\phi_{d,r}^l} \\
P_{d,\sigma}\ar@{^(->}[r]^-{e_{d,\sigma}} & Q_{d,\sigma}^l
}
\]
We have that $\Img(\phi_{d,r,\sigma})=Y_{d,r,\sigma}$, that $\Img(1\times e_{r,\sigma})=Q_{d,r}\times_{Q_r}E_{r,\sigma}^l$, and that $\phi_{d,r}^l$ maps $Q_{d,r}\times_{Q_r}E_{r,\sigma}^l$ onto $E_{d,\sigma}^l\cap Z_{d,r}^l$. It follows that $e_{d,\sigma}$ maps $Y_{d,r,\sigma}$ onto $E_{d,\sigma}^l\cap Z_{d,r}^l$.

Next, we show that $e_{d,\sigma}: P_{d,\sigma}\embed Q_d^l$ maps $\mathring{Y}_{d,r,\sigma}$ isomorphically onto $E_{d,\sigma}^l\cap \mathring{Z}_{d,r}^l$ for any $r$, $l<r<d$. We have a commutative diagram
\[
\xymatrix@C=2pc{
Q_{d,r}\times_{Q_{r}}\mathring{P}_{r,\sigma}\ar@{^(->}[r]^-{1\times e_{r,\sigma}}\ar@{^(->}[d]_{\phi_{d,r,\sigma}} & Q_{d,r}\times_{Q_{r}}(Q_{r}^l\setminus Z_{r,r-1}^l)\ar[d]_{\phi_{d,r}^l}\ar@{=}[r] & Q_{d,r}\times_{Q_r}(Q_r^{r-1}\setminus \ds\bigcup_{i=l+1}^{r-1} Z_{r,i}^{r-1})\ar@{^(->}[d]^{\phi_{d,r}^{r-1}} \\
P_{d,\sigma}\setminus Y_{d,r-1,\sigma}\ar@{^(->}[r]^{e_{d,\sigma}} & Q_d^l\setminus Z_{d,r-1}^l \ar@{=}[r] &
Q_d^{r-1}\setminus \ds\bigcup_{i=l+1}^{r-1}Z_{d,i}^{r-1}
}
\]
We see that $e_{r,\sigma}$ maps $\mathring{P}_{r,\sigma}$ isomorphically onto $E_{r,\sigma}^l\setminus Z_{r,r-1}^l$ in $Q_r^l\setminus Z_{r,r-1}^l$, that $E_{r,\sigma}^l\setminus Z_{r,r-1}^l$ can be identified with $E_{r,\sigma}^{r-1}\subset Q_r^{r-1}\setminus\bigcup_{i=l+1}^{r-1} Z_{r,i}^{r-1}$ {(by Lemma \ref{lem_E})}, and that $\phi_{d,r}^{r-1}$ maps $Q_{d,r}\times_{Q_r}E_{r,l}^{r-1}$ isomorphically onto $E_{d,l}^{r-1}\cap Z_{d,r}^{r-1}$, which is identified with $E_{d,l}^l\cap \mathring{Z}_{d,r}^l$. On the other hand, $\phi_{d,r,\sigma}$ maps $Q_{d,r}\times_{Q_{r}}\mathring{P}_{r,\sigma}$ in isomorphically onto $\mathring{Y}_{d,r,\sigma}\subset P_{d,\sigma}\setminus Y_{d,r-1,\sigma}$. It follows that $e_{d,\sigma}$ maps $\mathring{Y}_{d,r,\sigma}$ isomorphically onto $E_{d,\sigma}^l\cap \mathring{Z}_{d,r}^l$. Thus the case that $\sigma\in\Sigma_{m+1}$ is constructed.
\end{proof}

\begin{prop}
Let $\sigma\in\Sigma$. Then for any $d$ and $l$, $d>l\geq \lt(\sigma)$, there is an isomorphism $P_{d,\sigma}\setminus Y_{d,l,\sigma}\stackrel{\sim}{\to} E_{d,\sigma}^l$ which maps $\mathring{Y}_{d,r,\sigma}$ isomorphically onto $E_{d,\sigma}^l\cap \mathring{Z}_{d,r}^l$ for any $r$, $l<r<d$;
\end{prop}
\begin{proof}
The case that $l=\lt(\sigma)$ is proved in the above proposition. We now prove the case that $l>\lt(\sigma)$. Let $t=\lt(\sigma)$. Then by the above proposition, we have an isomorphism $P_{d,\sigma}\stackrel{\sim}{\to} E_{d,\sigma}^t$ which maps $Y_{d,l,\sigma}$ onto $E_{d,\sigma}^t\cap Z_{d,l}^t$ for each $r$. This isomorphism restricts to an isomorphism $P_{d,\sigma}\setminus Y_{d,l,\sigma}\stackrel{\sim}{\to} E_{d,\sigma}^t\setminus Z_{d,l}^t$. Under the identification $Q_d^t\setminus Z_{d,l}^t=Q_d^l\setminus \bigcup_{i=t+1}^l Z_{d,i}^l$, $E_{d,\sigma}^t\setminus Z_{d,l}^t$ is identified with $E_{d,\sigma}^l$, and $E_{d,\sigma}^t\cap \mathring{Z}_{d,r}^t$ is identified with $E_{d,\sigma}^l\cap \mathring{Z}_{d,r}^l$ for any $r$, $l<r<d$. Hence we obtain an isomorphism $P_{d,\sigma}\setminus Y_{d,l,\sigma}\stackrel{\sim}{\to} E_{d,\sigma}^l$ which maps $\mathring{Y}_{d,r,\sigma}$ isomorphically onto $E_{d,\sigma}^l\cap \mathring{Z}_{d,r}^l$ for any $r$, $l<r<d$.
\end{proof}

Let $\sigma=(l_m,\cdots,l_1)\in\Sigma$ and $d>l_m$. We have the following commutative diagram
\[
\CD{
& E_{d,\sigma}^{l_m}\ar[d] \\
P_{d,\sigma}\ar[ru]^{\simeq}\ar[d] &E_{d,\tau}^{l_{m-1}}\cap \mathring{Z}_{d,l_m}^{l_{m-1}}\ar@{^(->}[r] & E_{d,\tau}^{l_{m-1}}\ar[d] \\
\mathring{R}_{d,l_m,\tau}\ar[ru]^{\simeq}\ar@{^(->}[r] & P_{d,\tau}\ar[ru]^{\simeq}\ar[d] & \qquad\ddots\ar@{^(->}[r] & E_{d,l_3,l_2,l_1}^{l_3}\ar[d] \\
&\qquad\ddots\ar@{^(->}[r] & P_{d,l_3,l_2,l_1}\ar[d]\ar[ru]^{\simeq} & E_{d,l_2,l_1}^{l_2}\cap \mathring{Z}_{d,l_3}^{l_2}\ar@{^(->}[r] & E_{d,l_2,l_1}^{l_2}\ar[d] \\
& & \mathring{R}_{d,l_3,l_2,l_1}\ar@{^(->}[r]\ar[ru]^{\simeq} & P_{d,l_2,l_1}\ar[d]\ar[ru]^{\simeq} & E_{d,l_1}^{l_1}\cap \mathring{Z}_{d,l_2}^{l_1}\ar@{^(->}[r] & E_{d,l_1}^{l_1}\ar[d] \\
& & & \mathring{R}_{d,l_2,l_1}\ar@{^(->}[r]\ar[ru]^{\simeq} & P_{d,l_1}\ar[d]\ar[ru]^{\simeq} & \mathring{Z}_{d,l_1}\ar@{^(->}[r] & Q_d \\
& & & & \mathring{R}_{d,l_1}\ar[ru]^{\simeq}\ar@{^(->}[r]_{\mathring{\phi}_{d,l_1}} & P_d \ar@{=}[ru]
}
\]
where $\tau=(l_{m-1},\cdots,l_1,l_0)$. Thus we obtain a sequence of canonical identifications:
\[
\mathring{R}_{d,l_0}=\mathring{Z}_{d,l_0},\quad P_{d,l_0}=E_{d,l_0}^{l_0},\quad \mathring{R}_{d,l_1,l_0}=E_{d,l_0}^{l_0}\cap \mathring{Z}_{d,l_1}^{l_0},\quad \cdots,\quad P_{d,\sigma}=E_{d,\sigma}^{l_m}.
\]

Recall that $Q_d$ parametrizes quotient of the form $V_{\mathbb{P}^1}\surj X_1$ with $\deg X_1=d$, and the subset $\mathring{Z}_{d,l_1}$ of $Q_d$ parametrizes such quotients with $\deg X_1^\tor=d-l_1$.

Let $x_1=[V_{\mathbb{P}^1}\surj X_1]\in\mathring{Z}_{d,l_1}=\mathring{R}_{d,l_1}$. Using the identification,
$$E_{d,l_1}^{l_1}=P_{d,l_1}=\mathbb{P}(\pi_*\EExt^1(\bar{\mathring{q}}^*_{d,l_1}\mathcal{T}_{d,l_1}, \bar{\mathring{p}}_{d,l_1}^*\mathring{\mathcal{X}}_{l_1})),$$  we see that the fiber of $E_{d,l_1}^{l_1}\to \mathring{Z}_{d,l_1}$ over $x_1$ is $\mathbb{P}(\Ext^1((\bar{\mathring{q}}^*_{d,l}\mathcal{T}_{d,l})_{x_1}, (\bar{\mathring{p}}_{d,l}^*\mathring{\mathcal{X}}_{l})_{x_1}))$. On the other hand, we have an exact sequence
\[
0\to \bar{\mathring{q}}^*_{d,l_1}\mathcal{T}_{d,l_1}\to \bar{\mathring{\phi}}^*_{d,l_1}\mathcal{X}_d  \to \bar{\mathring{p}}_{d,l_1}^*\mathring{\mathcal{X}}_{l_1}\to 0
\]
whose restriction to $\mathbb{P}^1\times\{x_1\}$,
\[
0\to (\bar{\mathring{q}}^*_{d,l_1}\mathcal{T}_{d,l_1})_{x_1}\to (\bar{\mathring{\phi}}^*_{d,l_1}\mathcal{X}_d)_{x_1}  \to (\bar{\mathring{p}}_{d,l_1}^*\mathring{\mathcal{X}}_{l_1})_{x_1}\to 0
\]
is also an exact sequence. Since $\mathring{\phi}_{d,l_1}$ is an inclusion map, we have that $(\bar{\mathring{\phi}}^*_{d,l_1}\mathcal{X}_d)_{x_1}=X_1$. Hence $(\bar{\mathring{q}}^*_{d,l_1}\mathcal{T}_{d,l_1})_{x_1}=X_1^\tor$ and $(\bar{\mathring{p}}_{d,l_1}^*\mathring{\mathcal{X}}_{l_1})_{x_1}=X_1^\fre$, and the fiber over $x_1$ is $\mathbb{P}(\Ext^1(X_1^\tor,X_1^\fre))$, which parametrizes non-split extensions of the form $[X_1^\fre\inj X_2\surj X_1^\tor]$. Thus $E_{d,l_1}^{l_1}$ parametrizes sequences of the form
\[
([V_{\mathbb{P}^1}\surj X_1],[X_1^\fre\inj X_2\surj X_1^\tor]).
\]
with $\deg X_1^\tor=d-l_1$ and $\deg X_2^\tor<d-l_1$. Using the identification $\mathring{R}_{d,l_2,l_1}=\mathring{Y}_{d,l_2,l_1}=E_{d,l_1}^{l_1}\cap \mathring{Z}_{d,l_2}^{l_1}$, we see that $E_{d,l_1}^{l_1}\cap \mathring{Z}_{d,l_2}^{l_1}$ parametrizes such sequences with $\deg X_1^\tor=d-l_1$ and $\deg X_2^\tor =d-l_2$, by the definition of $\mathring{Y}_{d,l_2,l_1}$.

Let $x_2=([V_{\mathbb{P}^1}\surj X_1],[X_1^\fre\inj X_2\surj X_1^\tor])\in E_{d,l_1}^{l_1}\cap \mathring{Z}_{d,l_2}^{l_1}=\mathring{R}_{d,l_2,l_1}$. Using the identification $E_{d,l_2,l_1}^{l_2}=P_{d,l_2,l_1}=\mathbb{P}(\EExt^1({\bar{\mathring q}}_{d,l_2,l_1}^*\mathcal{T}_{d,l_2},{\bar{\mathring p}}_{d,l_2,l_1}^*\mathring{\mathcal{X}}_{l_2,l_1}))$, we have that
the fiber of $E_{d,l_2,l_1}^{l_2}\to E_{d,l_1}^{l_1}\cap \mathring{Z}_{d,l_2}^{l_1}$ over the point $x_2$ is $\mathbb{P}(\EExt^1(({\bar{\mathring q}}_{d,l_2,l_1}^*\mathcal{T}_{d,l_2})_{x_2},({\bar{\mathring p}}_{d,l_2,l_1}^*\mathring{\mathcal{X}}_{l_2,l_1})_{x_2}))$. On the other hand, we have an exact sequence
\[
0\to {\bar{\mathring q}}_{d,l_2,l_1}^*\mathcal{T}_{d,l_2}\to \bar{\mathring{\phi}}_{d,l_2,l_1}^*\mathcal{X}_{d,l_1}\to {\bar{\mathring p}}_{d,l_2,l_1}^*\mathring{\mathcal{X}}_{l_2,l_1}\to 0
\]
whose restriction to $\mathbb{P}^1\times\{x_2\}$,
\[
0\to ({\bar{\mathring q}}_{d,l_2,l_1}^*\mathcal{T}_{d,l_2})_{x_2}\to (\bar{\mathring{\phi}}_{d,l_2,l_1}^*\mathcal{X}_{d,l_1})_{x_2}\to ({\bar{\mathring p}}_{d,l_2,l_1}^*\mathring{\mathcal{X}}_{l_2,l_1})_{x_2}\to 0
\]
is also an exact sequence. Since $\bar{\mathring{\phi}}_{d,l_2,l_1}^*$ is an inclusion map, we have $(\bar{\mathring{\phi}}_{d,l_2,l_1}^*\mathcal{X}_{d,l_1})_{x_2}=X_2$. Hence $({\bar{\mathring q}}_{d,l_2,l_1}^*\mathcal{T}_{d,l_2})_{x_2}=X_2^\tor$, $({\bar{\mathring p}}_{d,l_2,l_1}^*\mathring{\mathcal{X}}_{l_2,l_1})_{x_2}=X_2^\fre$, and the fiber over $x_2$ is $\mathbb{P}(\Ext^1(X_2^\tor, X_2^\fre))$, which parametrizes non-split extensions of the form $[X_2^\fre\inj X_3\surj X_2^\tor]$. Thus $E_{d,l_2,l_1}^{l_2}$ parametrizes sequences of the form:
\[
([V_{\mathbb{P}^1}\surj X_1],[X_1^\fre\inj X_2\surj X_1^\tor],[X_2^\fre\inj X_3\surj X_2^\tor]).
\]
with $\deg X_1^\tor=d-l_1$, $\deg X_2^\tor=d-l_2$ and $\deg X_3^\tor<d-l_2$. Since $\mathring{R}_{d,l_3,l_2,l_1}=\mathring{Y}_{d,l_3,l_2,l_1}=E_{d,l_2,l_1}^{l_2}\cap \mathring{Z}_{d,l_3}^{l_2}$, by the definition of $\mathring{Y}_{d,l_3,l_2,l_1}$, $E_{d,l_2,l_1}^{l_2}\cap \mathring{Z}_{d,l_3}^{l_2}$ parametrizes such sequences with $\deg X_1^\tor=d-l_1$, $\deg X_2^\tor=d-l_2$ and $\deg X_3^\tor=d-l_3$.

Continuing this argument, we eventually obtain the parameter-space interpretation for $E_{d,\sigma}^{l_m}$:
$E_{d,\sigma}^{l_m}$ parametrizes sequences of the form
\[
([V_{\mathbb{P}^1}\surj X_1],[X_1^\fre\inj X_2\surj X_1^\tor],\cdots, [X_{m}^\fre\inj X_{m+1}\surj X_{m}^\tor])
\]
with $\deg X_i^\tor=d-l_i$, $(i=1,\cdots,m)$ and $\deg X_{m+1}^\tor <d-l_m$. For any $r$ with $d>r>l_m$, $E_{d,\sigma}^{l_m}\cap \mathring{Z}_{d,r}^{l_m}$ parametrizes such sequences with  $\deg X_i^\tor=d-l_i$, $(i=1,\cdots,m)$ and $\deg X_{m+1}^\tor =d-r$.

Next, we deal with the modular interpretation of $E_{d,\sigma}^r$ for any $r$, $d>r>l_m$. By the lemma before, we have $E_{d,\sigma}^r=E_{d,\sigma}^{l_m}\setminus Z_{d,r}^{l_m}$. Since $Z_{d,r}^{l_m}=\bigcup_{i=l_m+1}^r \mathring{Z}_{d,i}^{l_m}$, we know that $E_{d,\sigma}^{l_m}\cap Z_{d,r}^{l_m}$ parametrizes sequences as above with $\deg X_i^\tor=d-l_i$, $(i=1,\cdots,m)$ and $\deg X_{m+1}^\tor \geq d-r$. Therefore, $E_{d,\sigma}^r$, which equals $E_{d,\sigma}^{l_m}\setminus Z_{d,r}^{l_m}$, parametrizes such sequences with $\deg X_i^\tor=d-l_i$, $(i=1,\cdots,m)$ and $\deg X_{m+1}^\tor < d-r$.

In summary, we have
\begin{prop}\label{stratum}
Let $d>r$. For any $\sigma=(l_m,\cdots,l_2,l_1)\in\Sigma$ with $l_m<r$, the stratum $E_{d,\sigma}^r$ of $Q_d^r$ parametrizes the sequences of the form
\[
([V_{\mathbb{P}^1}\surj X_1],[X_1^\fre\inj X_2\surj X_1^\tor],\cdots, [X_{m}^\fre\inj X_{m+1}\surj X_{m}^\tor])
\]
with $\deg X_i^\tor=d-l_i$, $(i=1,\cdots,m)$ and $\deg X_{m+1}^\tor < d-r$.
\end{prop}

\begin{cor}\label{final-interp}
The boundary $\widetilde Q_d\setminus \mathring{Q}_d$ comes equipped with a natural stratification with strata
$E_{d,\sigma}$  indexed by  $\sigma=(l_m,\cdots,l_2,l_1)$ with $d>l_m >\dots >l_1 \geq 0$,.
The stratum $E_{d,\sigma}$  parametrizes the sequences of the form
\[
([V_{\mathbb{P}^1}\surj X_1],[X_1^\fre\inj X_2\surj X_1^\tor],\cdots, [X_{m}^\fre\inj X_{m+1}\surj X_{m}^\tor])
\]
such that  $\deg X_i^\tor=d-l_i$, $i=1,\cdots,m$, and  the last sheaf $X_{m+1}$ is the unique one that is  locally free.
\end{cor}
This proves Theorem \ref{hs}.

\end{document}